\documentclass[12pt]{amsart}
\usepackage[myheadings]{fullpage}
\usepackage{fancyhdr}
\usepackage{amsmath,amssymb,amsbsy,amsthm,amsfonts,epsfig,verbatim,enumerate,comment}
\usepackage{graphicx,xcolor,color,wrapfig}
\usepackage{pxfonts,bm}
\usepackage{hyperref}
\usepackage{pbox}
\usepackage{centernot}
\usepackage{ stmaryrd }

\input{xy}
\xyoption{all}
\CompileMatrices

\newcommand{\ed}{{\rm d}}
\newcommand{\w}{{\mathchoice{\,{\scriptstyle\wedge}\,}{{\scriptstyle\wedge}}
      {{\scriptscriptstyle\wedge}}{{\scriptscriptstyle\wedge}}}}

\newcommand{\del}{{\partial}}

\newcommand{\delx}{\partial_{\xi}}
\newcommand{\delxb}{\partial_{\xib}}

\newcommand{\gap}[1]{\vsp{#1pc}}

\newcommand{\liealgebra}[1]{{\mathfrak {#1}}}
\newcommand{\g}{\liealgebra{g}}

\newcommand{\sla}{\liealgebra{sl}}
\newcommand{\so}{\liealgebra{so}}

\newcommand{\liegroup}[1]{{\operatorname{#1}}}

\newcommand{\SO}{\liegroup{SO}}

\newcommand{\R}{\mathbb R}
\newcommand{\C}{\mathbb C}

\newcommand{\Z}{\mathbb Z}

\newcommand{\PP}{\mathbb P}
\newcommand{\tr}{\rm tr}
\newcommand{\dd}[2]{\frac{\partial {#1}}{\partial {#2}}}
\newcommand{\bdot}{\overset{\bm .}}

\usepackage{tikz}

\newcommand{\mcd}{\mathcal D}
\newcommand{\mce}{\mathcal E}
\newcommand{\mcf}{\mathcal F}

\newcommand{\mci}{\mathcal I}

\newcommand{\mcl}{\mathcal L}
\newcommand{\mcm}{\mathcal M}

\newcommand{\mcr}{\mathcal R}

\newcommand{\mcv}{\mathcal V}

\newcommand{\nB}{\textnormal{B}}

\newcommand{\ES}{\textnormal{S}}
\newcommand{\x}{\textnormal{x}}

\newcommand{\im}{\textnormal{i}}

\newcommand{\bpsi}{\boldsymbol{\psi}}
\newcommand{\bphi}{\boldsymbol{\phi}}

\newcommand{\bsigma}{\boldsymbol{\sigma}}

\newcommand{\bba}{\tb{a}}
\newcommand{\bbb}{\tb{b}}
\newcommand{\bbc}{\tb{c}}

\newcommand{\bG}{\tb{G}}

\newcommand{\bbt}{\tb{t}}

\newcommand{\bX}{\tb{X}}
\newcommand{\bY}{\tb{Y}}

\newcommand{\bmg}{\boldsymbol{\g}} 

\newcommand{\ol}{\overline}

\newcommand{\zb}{\ol{z}}

\newcommand{\hb}{\bar h}

\newcommand{\ab}{\ol{a}}

\newcommand{\etab}{\ol{\eta}}

\newcommand{\tba}{\ol{t}}

\newcommand{\xib}{\ol{\xi}}

\newcommand{\nR}{\tn{R}}

\newcommand{\Fh}[1]{\hat{\mcf}^{(#1)}}

\newcommand{\Xh}[1]{\hat{X}^{(#1)}}

\newcommand{\Ih}[1]{{\hat{\rm I}}^{(#1)}}

\newcommand{\xinf}{X^{(\infty)}}
\newcommand{\xinfh}{\hat{X}^{(\infty)}}
\newcommand{\iinf}{{\rm I}^{(\infty)}}
\newcommand{\iinfh}{\hat{\rm I}^{(\infty)}}

\newcommand{\finf}{{\mcf}^{(\infty)}}
\newcommand{\finfh}{{\hat{\mcf}}^{(\infty)}}

\newcommand{\be}{\begin{equation}}
\newcommand{\ee}{\end{equation}}
\newcommand{\benu}{\begin{enumerate}}
\newcommand{\enu}{\end{enumerate}}
\newcommand{\beit}{\begin{itemize}}
\newcommand{\enit}{\end{itemize}}

\newcommand{\bp}{\begin{pmatrix}}
\newcommand{\ep}{\end{pmatrix}}

\newcommand{\lra}{\longrightarrow}

\newcommand{\ff}{{\rm I \negthinspace I}}

\newcommand{\n}{\notag}
\newcommand{\noi}{\noindent}

\newcommand{\tn}{\textnormal}
\newcommand{\tb}{\textbf}
\newcommand{\vsp}{\vspace}
\newcommand{\hook}{\hookrightarrow}
\newcommand{\marg}{\marginpar}
\newcommand{\one}{\vspace{1mm}}
\newcommand{\two}{\vspace{2mm}}
\newcommand{\twoline}{\two \begin{center}\line(1,0){100}\end{center}\two\one}

\newcommand{\ftmark}{\footnotemark}
\newcommand{\fttext}{\footnotetext}
\newcommand{\sub}{\subsection}
\newcommand{\subb}{\subsubsection}

\makeatletter

\newcommand{\Rmnum}[1]{\expandafter\@slowromancap\romannumeral #1@}
\makeatother

\usepackage[mathscr]{eucal}
\usepackage{bbm}
\usepackage{oldgerm}

\theoremstyle{plain}
\newtheorem{thm}{Theorem}[section]
\newtheorem{lem}[thm]{Lemma}

\newtheorem{cor}[thm]{Corollary}

\newtheorem{prop}[thm]{Proposition}

\theoremstyle{definition}
\newtheorem{defn}{Definition}[section]

\begin{document}

%%%%%%%%%%%%%%%%%%%%%%%%%%%%%%%%%%%%%%%%%
\title[CMC hierarchy \Rmnum{1}]
{CMC hierarchy \Rmnum{1}: Commuting symmetries and loop algebra}
\author{Joe S. Wang}
%\address{Seoul, South Korea}
\email{jswang12@gmail.com}
%53C43, 37K10
\date{\today}
%differential geometry, exterior differential system,  minimal Lagrangian surface, characteristic cohomology, Jacobi field, conservation law,  formal Killing field, recursion}  %Integrable extension, Non-local symmetry, Secondary characteristic cohomology, Affine algebra, Noether's theorem}
%\thanks{
%}
\begin{abstract} 
We propose an extension of the structure equation
for constant mean curvature (CMC) surfaces in a three dimensional Riemannian space form
to the associated CMC hierarchy of evolution equations
by the higher-order commuting symmetries.
Via the canonical formal Killing field, considered as 
an infinitely prolonged and loop algebra valued Gau\ss\, map,
the CMC hierarchy is obtained by the assembly of a pair of 
Adler-Kostant-Symes bi-Hamiltonian hierarchies  
to the original CMC system.
The infinite sequence of higher-order conservation laws of the CMC system 
admits the corresponding extension,
and we find a formula for the generating series of the representative 1-forms.
We also introduce a class of generalized (complexified) CMC surfaces  
as the phase space of the CMC hierarchy.
\end{abstract}
\maketitle

%-----------------------------------------------------------------------------%
\setcounter{tocdepth}{2}
\tableofcontents

%-----------------------------------------------------------------------------%
%\listoftables   %if you have any tables

%-----------------------------------------------------------------------------%
%\listoffigures  %if you have any figures

%-----------------------------------------------------------------------------%
%
% MAIN BODY OF PAPER
%
% replace TEXT in \chapter{TEXT} with actual chapter name
% replace FILE in \input{FILE} with name of text file
%   containing the given chapter, eg. for the introduction one could
%   have FILE = intro IF stored in intro.tex (.tex extension is assumed!).
%
%%%%%%%%%%%%%%%%%%%%%%%%%%%%%%%%%%%%

%%%%%%%%%%
%%%%%%%%%%
%%%%%%%%%%
%%%%%%%%%% 
%%%%%%%%%%%%%%%%%%%%
%\setcounter{section}{0}
\section{Introduction}\label{sec:intro}
%---------------------------------------------------------
%---------------------------------------------------------
\sub{Symmetry of a differential equation}
%---------------------------------------------------------
\subb{Classical symmetry}
The classical notion of symmetry of a differential equation is
a change of variable;
a change of dependent and independent variables which maps solutions to solutions.
From the inception of Lie groups and Lie algebras, 
Lie himself viewed a Lie group as a transformation group of symmetries
of a differential equation, and more generally of  a differential geometric object,
\cite{Lie1880,Lie1884} translated in \cite{Hermann1975,Hermann1976}. 

On the other hand, a classical symmetry
satisfies the uniform jet-order constraints and
it is generated by the (infinitesimal) transformations defined 
on a finite jet space.
%---------------------------------------------------------
\subb{Generalized symmetry}
A conceptual working definition of the symmetry of a differential equation would be:
\two

\begin{center}
\emph{a (local) Lie group or a Lie algebra which acts on the (formal) moduli space of solutions.}
\end{center}

\two\noi
In this generalized sense, many new forms of symmetries,  
which are rooted in the deeper structural properties of a differential equation,
become available.
For example, one of the initial discoveries regarding the integrable hierarchies 
was that a differential equation may admit another compatible (commuting) evolution equation 
as a symmetry, \cite{Mulase1994}\cite{DJT2000}\cite{Dickey2003} and the references therein.
%---------------------------------------------------------
\subb{Infinite prolongation and integrable extension}
There exist many differential equations of interest
which admit the various kinds of  generalized, higher-order symmetries.
To accommodate these,
it is necessary to consider the infinite jets of a differential equation as a whole,
and the infinite prolongation space  provides the adequate background for analysis.

Furthermore, there are occasions when it is relevant to introduce 
the auxiliary non-local variables by integrable extension,
\cite{Krasilshchik2004}. Simply stated, this amounts to supplementing the given differential equation
with an additional system of compatible ODE's.

%\two
%Consider the following analogy with the algebraic equation case.
 
\begin{table}[h] \centering
  \begin{tabular}{ l | l  }
Differential equation&Algebraic equation  \\
    \hline\hline
\pbox{8cm}{integrable extension} & 
\pbox{8cm}{field extension}    \\ \hline
\pbox{8cm}{infinite  prolongation} &  
\pbox{8cm}{completion}   \\  
  \hline 
\end{tabular}
\two\two
  \caption{Analogy with  algebraic equation}\label{tab:one}
\end{table} 
\vspace{-6mm}

In this extended setting, the space of symmetries corresponds
to the kernel of the linearization of the infinitely prolonged differential equation.
The foundational works of  Tsujishita \cite{Tsujishita1982},  Vinogradov \cite{Vinogradov19841,  Vinogradov19842}, and Bryant and Griffiths \cite{Bryant1995} 
provide the general methods of commutative algebraic analysis 
to compute the symmetries and  other cohomological invariants of a differential equation.

Regarding the integrable extension, we mention for an example that 
the log of tau function of   KP hierarchy 
is defined as the potential for a non-local closed 1-form obtained by dressing,
\cite{Dickey2003}.
%---------------------------------------------------------
%---------------------------------------------------------
\sub{Symmetry extension}%horizontal
One of the characteristic defining properties of an integrable equation is that  
it admits an infinite sequence of higher-order commuting flows (evolution equations)
as symmetries.
This in turn leads to the extension of the given differential equation to  
the associated infinite hierarchy of equations.
Compared to infinite prolongation and integrable extension,
which are vertical extensions in a sense,
the symmetry extension of a differential equation, 
which increases the number of independent variables,
can be considered as a horizontal extension. 

\gap{1.7}

\begin{figure}[h]
\centerline{
\xymatrix@C=2cm{\tn{generalized symmetries} \ar@/^7.2mm/[r] & \ar@/^7mm/[l] 
\tn{symmetry extension}}
}\two\two
\label{fig:extension}
\caption{Symmetry iteration}
\end{figure}
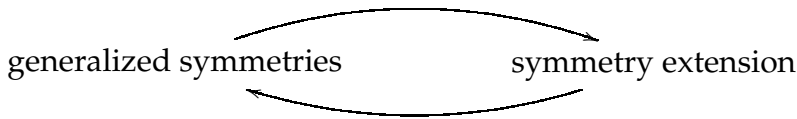
 
An important consequence  of the symmetry extension
is that it may lead to the additional symmetries of a differential equation.
By iterating the two processes of finding  generalized symmetries and  symmetry extension,
one ultimately hopes to gain an insight into \emph{solving} the given differential equation.
Note  that 
the stationary solutions to the additional symmetries of an integrable equation
provide a new class of solutions different from the  finite type, algebro-geometric solutions, \cite{Moerbeke1994}\cite{Mulase1999}\cite{Safronov2013}.
%---------------------------------------------------------
%---------------------------------------------------------
\subsection{CMC hierarchy} 
The elliptic Monge-Ampere system for constant mean curvature (CMC) surfaces 
in a three dimensional Riemannian space form 
is the typical example of an integrable elliptic equation;
in particular,  it possesses an infinite sequence of higher-order symmetries 
and conservation laws, \cite{Wang2013}.
We wish to apply the idea of symmetry extension described above
to find the additional symmetries,
and subsequently to understand their stationary solutions.

In this first part of the series on CMC hierarchy, we will discuss 
the infinite sequence of commuting symmetries of the CMC system
which are based on a twisted loop algebra $\bmg\subset \sla(2,\C)((\lambda))$, \eqref{eq:twisted}.
%$\ftmark, \fttext{Here $\sla(2,\C)((\lambda))$ is 
%the loop algebra of $\sla(2,\C)$-valued formal Laurent series in the spectral parameter $\lambda$.}
%---------------------------------------------------------
\subb{Purpose}
The purpose of this paper is to show that the CMC system
admits a symmetry extension by the higher-order commuting  symmetries
to the  compatible CMC hierarchy.
%As expected, this will lead to a pair of additional non-local formal Killing fields and
%an associated spectral Killing field\footnotemark.\footnotetext{This corresponds 
%to a Virasoro algebra type of non-commuting symmetries.
%It will be treated in Part $\ff$ of the series.} 
Although in a slightly different context,
the general model for our investigation is  the Frenkel's work \cite{Frenkel1998}
on  Drinfeld-Sokolov hierarchies.

\iffalse
Let us consider the KdV equation for an example of integrable equation.
In this case, by embedding the formal moduli space of the scalar functions in 1 variable
to the formal moduli space of the one dimensional second order Schr\"odinger operators,
the compatibility of the higher order jet structure of the scalar functions
with the Lie algebra structure of the pseudo-differential operators
induces an infinite sequence of higher-order commuting flows (symmetries) 
on the formal moduli space of the scalar functions.
The analogy with the present CMC system can be summarized as follows.

\two
\begin{center}
  \begin{tabular}{ l | l  }
KdV& CMC  \\
    \hline\hline
\pbox{7cm}{formal moduli space of \\ scalar functions in 1 variable} &  
\pbox{7cm}{formal moduli space of \\  generalized  CMC surfaces}   \\ \hline
\pbox{7cm}{Lie algebra of \\ pseudo differential operators} & 
\pbox{7cm}{loop algebra \\ $\sla(2,\C)[[\lambda]]$}    \\
    \hline
  \end{tabular}
\end{center}
\two\one
\fi

%%%%%%%%%%%%%%%%%
\subb{CMC hierarchy}
We will find that, as a system of PDE's, 
the proposed CMC hierarchy is locally equivalent to,
\[ -\ol{\tn{mKdV hierarchy}}^t  \oplus \tn{elliptic sinh-Gordon}\oplus\tn{mKdV hierarchy.} 
\]
For a related work on sine Gordon $\oplus$ mKdV hierarchy, we refer to \cite{Gesztesy2000}.
 
From this, it is expected that a substantial part of the existing theory of integrable systems
can be introduced to the study of CMC surfaces. 

%---------------------------------------------------------
%---------------------------------------------------------
\subsection{Results}
In the previous work \cite{Wang2013}, 
we gave a differential algebraic inductive formula  
for the $\bmg_{\geq 1}$-valued\ftmark\fttext{Here $\bmg_{\geq 1}=
\bmg\cap\sla(2,\C)[[\lambda]]\lambda$ is 
the subalgebra of formal power series of $\lambda$-degree $\geq 1$, 
\eqref{eq:vsdecomp}, \eqref{eq:vsdecompdual}.}   
 canonical formal Killing field, denoted by $\bY$.
The Jacobi fields and conservations laws of the CMC system were embedded 
in  the coefficients of $\bY$,  and  accordingly 
the infinite sequence of higher-order Jacobi fields and conservation laws were completely determined.

The algebraic basis of this results lies in 
the compatibility of  the prolongation structure of the CMC system with 
the recursive structure equation of the loop algebra $\bmg_{\geq 1}$.
We claim that 
the consequences of this compatibility go  beyond 
the effective calculation of  the Jacobi fields and conservation laws.
%---------------------------------------------------------
\subb{Complexified CMC surface}
We  introduce a class of complexified CMC surfaces as a generalization of the ordinary CMC surfaces,
Defn.\ref{defn:gCMC}.
They serve as  the phase space of the CMC hierarchy.

%---------------------------------------------------------
\subb{Infinitely prolonged and loop algebra valued Gau\ss\, map}
%Recall the loop algebra $\bmg$, \eqref{eq:twisted}.
The  canonical formal Killing field $\bY$ induces a map,
$$\Fh{\infty}_{**}\xrightarrow{\quad(-\ol{\bY}^t,\bY)\quad}-\ol{\bmg_{\geq 1}}^t\times\bmg_{\geq 1},$$
which can be considered as an infinitely prolonged version of Gau\ss\, map, 
Fig.\ref{fig:InfGauss2}, Defn.\ref{defn:InfGauss2}.
Here $\Fh{\infty}$ is roughly the infinite jet space of the CMC system, 
and $\Fh{\infty}_{**}\subset\Fh{\infty}$ is a certain Zariski open set, \S\ref{sec:Zariski}.

%---------------------------------------------------------
\subb{CMC hierarchy}
The twisted loop algebra $\bmg$ supports the Adler-Kostant-Symes (AKS) bi-Hamiltonian hierarchy,
which is induced from the vector space decomposition, \eqref{eq:vsdecomp},
$$\bmg=\bmg_{\leq -1}+\bmg_{\geq 0},$$
and the associated R-matrix, Defn.\ref{defn:AKS}.
The fundamental observation for the construction of the CMC hierarchy is that 
the image of the map $(-\ol{\bY}^t,\bY)$ is tangent to (or agrees with)
the first two flows of the pair  of AKS hierarchies on $(-\ol{\bmg_{\geq 1}}^t,\bmg_{\geq 1})$ respectively,
Lem.\ref{lem:critical}.
From the formal symmetry of the Maurer-Cartan form, \eqref{eq:phiformalsymm},
the CMC hierarchy is obtained by
attaching the pair of AKS hierarchies via $(-\ol{\bY}^t, \bY)$
to a combined system of equations on $-\ol{\bmg_{\geq 1}}^t\times\bmg_{\geq 1}$,
Thm.~\ref{thm:main}.

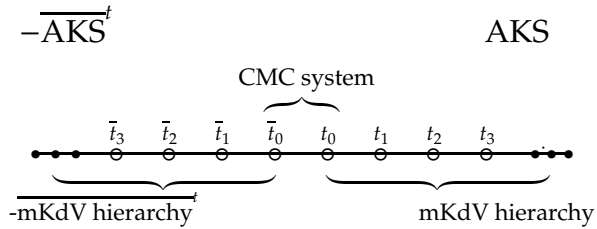
\begin{figure}[h!]
\begin{picture}(350, 70)( 10, 15)

\put(74,62){$-\ol{\tn{AKS}}^t$}
\put(250,62){AKS}
\put(166.4,32){$\overbrace{}$}
\put(190,15){$\underbrace{ \hspace{3cm} }$}
\put(86.4,15){$\underbrace{ \hspace{3cm} }$}
\put(157,46){\scriptsize{CMC system}  }  
\put(225,-5){\scriptsize{mKdV hierarchy}  }  
\put(70,-5){\scriptsize{-$\ol{\tn{mKdV hierarchy}}^t$}  }  

\put(80,20){\line(1,0){200}}
%\put(180, 20){\line(0,1){40}}
 
\multiput(107.5,16.4)(20,0){8}{$\circ$}
\put(78,17.5){\tiny{$\bullet$}\, \tiny{$\bullet$}\;\,\tiny{$\bullet$}}
\put(267,17.5){\tiny{$\bullet$}\.\tiny{$\bullet$}\;\tiny{$\bullet$}}

\put(168,25){\tiny{$\tba_0$}} 
\put(148,25){\tiny{$\tba_1$}} 
\put(128,25){\tiny{$\tba_2$}} 
\put(108,25){\tiny{$\tba_3$}} 
\put(188,25){\tiny{$t_0$}} 
\put(208,25){\tiny{$t_1$}} 
\put(228,25){\tiny{$t_2$}} 
\put(248,25){\tiny{$t_3$}} 
\end{picture}
\two\two\two

\caption{Anatomy of \;CMC hierarchy $\simeq  -\ol{\tn{AKS hierarchy}}^t$  $ \oplus$ AKS hierarchy}
\label{CMCdiagram}
\end{figure} 
 
\noi 
In this schematic picture\ftmark, $\{ \tba_m, t_m \}_{m\geq 0}$  are the time variables 
for the AKS hierarchies. 
\fttext{Strictly speaking, the CMC hierarchy is an integrable extension of 
 $-\ol{\tn{AKS hierarchy}}^t$  $ \oplus$  AKS hierarchy, Fig.\ref{fig:anatomy2}.}
Note that the original CMC system is embedded as the $\tba_0, t_0$-flows.
%---------------------------------------------------------
\subb{Extension of conservation laws}
The infinite sequence of higher-order conservation laws of the CMC system 
admits the corresponding extension to  the  CMC hierarchy.
We find an explicit formula for the generating series of the representative 1-forms, Thm.~\ref{thm:extcvlaw}. 
%---------------------------------------------------------
\subb{Linear finite type surfaces}
For an application, we show that
the linear finite type  CMC surfaces are characterized by the property
that the canonical formal Killing field $\bY$
is  stationary with respect to a higher-order symmetry, Cor.~\ref{cor:LFT}.
 
%---------------------------------------------------------
%---------------------------------------------------------
\sub{Contents}%\marg{check after revision}%---------------------------------------------------
After a summary of the  results from  \cite{Wang2013} in \S\ref{sec:summary},
we introduce in  \S\ref{sec:gCMC} a class of generalized CMC surfaces
as a complexification of the CMC surfaces.
%The formal moduli space of the complexified CMC surfaces will serve as the phase space of the CMC hierarchy.
In \S\ref{sec:YGauss}, 
we examine the algebraic properties of the $\bmg_{\geq 1}$-valued canonical formal Killing field,
considered as an infinitely prolonged version of Gau\ss\, map.
In \S\ref{sec:AKS},
the AKS construction of bi-Hamiltonian hierarchy is adapted to  the twisted loop algebra $\bmg$.
Based on this, we propose in \S\ref{sec:CMChierarchy2}
an ansatz for the CMC hierarchy in terms of 
an $\sla(2,\C)[[\lambda^{-1},\lambda]]$-valued\ftmark\fttext{Here $\sla(2,\C)[[\lambda^{-1},\lambda]]$ is
 the Lie algebra of $\sla(2,\C)$-valued formal power series in $\lambda^{-1}, \lambda$.} 
  extended Maurer-Cartan form.
In \S\ref{sec:CMChierarchy4},
the CMC hierarchy is translated into the $\so(4,\C)$-setting.
In \S\ref{sec:proof}, we give a proof by direct computation
that the proposed structure equation for the CMC hierarchy is compatible.
In \S\ref{sec:extcvlaw}, we show that there exists the corresponding extension of 
the infinite sequence of higher-order conservation laws.
In \S\ref{sec:LFT}, we give a geometric characterization 
of the linear finite type  CMC surfaces.

%---------------------------------------------------------
%---------------------------------------------------------
\sub{Remarks}
\subb{Extension sequence of the underlying Lie algebras}
In terms of the underlying Lie algebras only, the extension process of the CMC system up to the CMC hierarchy can be summarized as follows:
$$
\sla(2,\C)\lra\sla(2,\C)[\lambda^{-1},\lambda]\lra\sla(2,\C)[[\lambda^{-1},\lambda]].$$

The next step of the extension involves the Virasoro type of non-commuting symmetries (called spectral symmetries) and the associated generalized affine Kac-Moody algebras.
This will be reported in Part $\ff$ of the series.
%---------------------------------------------------------
\subb{Application}
The construction of compact, high genus CMC surfaces so far has relied on 
the analytic existence results of PDE's;
either to find the fundamental domains for reflection,
or to perturb an approximate CMC surface obtained by gluing to an actual CMC surface, 
 \cite{Brendle2013} for a survey of the related works.
One of the initial objectives of \cite{Wang2013} was to find 
a class of generalized (nonlinear) finite type CMC surfaces.
We hope that the stationary constraints from the additional symmetries 
may lead to a new class of CMC surfaces which can be analyzed by the methods of ODE's.

%%%%%%%%%%%%
%%%%%%%%%%%%   
%%%%%%%%%%%%
%%%%%%%%%%%%  
%%%%%%%%%%%%
%-------------------------------------------------
\section{Summary of previous results}\label{sec:summary}
We recall  the relevant notations and results from \cite{Wang2013}.
We only give a brief description  and refer the reader to \cite{Wang2013} for the details.
 
%--------------------------------------------------------------------
\subsection{Differential system}\label{sec:summary1} $\,$

\noi\one%-----------------------------------
[\tb{Grassmann bundle of oriented 2-planes}]
\begin{align}
M: &\; \tn{three dimensional Riemannian space form of constant curvature $\epsilon$,}\n\\
%    &\; \tn{of constant curvature $\epsilon$}, \n\\
\mcf:=\tn{Iso}(M): &\; \tn{group of isometries of $M$},\n\\
\quad X:=\tn{Gr}^+(2,TM):  &\; \tn{Grassmann bundle of oriented 2-planes}.\n 
\end{align}
They fit into the commutative diagram:

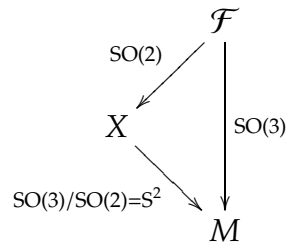
\two
\centerline{\xymatrix{ &&\mcf \ar[dl]_-{\SO(2)}\ar[dd]^{\SO(3)}  \\
  & X \ar[dr]_-{\SO(3)/\SO(2)=\ES^2}  &   \\ & &M}}\two\two

\noi%-----------------------------------
[\tb{Structure constant $\gamma$}]\newline
\indent
We shall consider the immersed oriented surfaces in $M$ of constant mean curvature  $\delta$.
\begin{align}
 \gamma^2:=\epsilon+\delta^2: &\; \tn{structure constant, \hspace{5pc}$\qquad$}\qquad\qquad\n\\
\tn{assumption}:  &\;  \tn{$\gamma^2>0$ and $\gamma$ is real}.\n
\end{align}
The case $\gamma^2<0$ appears to be incompatible with 
the certain aspects of the theory of integrable systems applied in this work.\two

\noi%-----------------------------------
[\tb{CMC system}]
\begin{align}
(X, \mci):    &\; \tn{original CMC system on $X$, $\qquad$ }\qquad\hspace{3pc}\qquad \n\\
(\xinf,\iinf): &\; \tn{infinite prolongation of $(X,\mci)$},\n\\
\finf\to\mcf: &\; \tn{pulled back bundle},\n\\
\finfh\to\finf, \xinfh \to\xinf  : &\; \tn{double covers}.\n
\end{align}
They fit into the commutative diagrams:
\be\label{eq:diagramss}
\xymatrix{ \finf \ar[d] \ar[r] & \xinf \ar[d]\\
\mcf \ar[r]^{\SO(2)} & X,
}\qquad\quad
\xymatrix{ \finfh \ar[d] \ar[r] & \xinfh \ar[d]\\
\finf \ar[r]^{\SO(2)} & \xinf.
}\ee
The corresponding differential ideals on $\finf$, and  $\xinfh, \finfh$ are denoted by 
$\iinf$, and $\iinfh$ respectively.

%--------------------------------------------------------------------
\subsection{Structure equations}\label{sec:summary2} $\,$
The structure equations recorded below are written modulo the appropriate differential ideals,
see \S\ref{sec:hierarchysystem} for a related remark.
The meaning is generally clear from the context, and we omit the specific descriptions.\two

\noi
[\tb{Basic structures}]
\begin{align}
\xi&:    \; \tn{tautological unitary (1,0)-form},\n\\
\rho&:  \; \tn{connection 1-form},\n\\
\ed\xi&=\im\rho\w\xi, \n \\
\ed\rho&= R\frac{\im}{2} \xi\w\xib,\n\\
\ff&=h_2\xi^2 :  \; \tn{Hopf differential},\n\\
R&=\gamma^2-h_2\hb_2 :  \; \tn{Gau\ss\, curvature}.\n
\end{align}

\noi
[\tb{Infinite prolongation}]
\begin{align}\label{eq:xstrt}
\ed h_j+\im j h_j \rho &= h_{j+1}\xi+ T_j \xib,  \quad j\geq 2,\n\\
T_2&=0, \n\\
T_{j+1}&=   \sum_{s=0}^{j-2}
a_{js} \,  h_{j-s} \, \partial^{s}_{\xi} R,    \; \; \; \mbox{for} \; \; j\geq 2,  \n   \\
& \quad  a_{js} =\frac{(j+s+2) }{2} \frac{ (j-1) !}{(j-s-2)!(s+2)!}
             =\frac{(j+s+2) }{2\, j} {j \choose s+2}, \n\\
& \;\, \partial^{s}_{\xi}  R=\delta_{0s}\gamma^2 -h_{2+s} \hb_2.\n
\end{align}
Here $\im=\sqrt{-1}$ denotes the unit imaginary number.

\iffalse 
The sequence of functions,
$$h_j: \finf\to\C\cup\{\infty\}, \quad j \geq 2,$$   
form a (extended) coordinate system when restricted to a fiber of the projection 
$\finf\to \mcf$.
\fi  
 
\two\noi
[\tb{$\sqrt{\ff}$, and balanced coordinates}]
\begin{align}
\omega:=\sqrt{\ff} &=h_2^{\frac{1}{2}}\xi,\n\\
\ed\omega&=0,\n\\
z_j&:=h_2^{-\frac{j}{2}}h_j, \quad j\geq 3,\n\\
\mcr &:=\C[z_3,z_4, \, ... \, ] , \quad 
\ol{\mcr}:=\C[\zb_3, \zb_4, \, ... \, ]. \n
\end{align}
Assign the spectral weights by,
\be\label{eq:spectralweight}
\begin{array}{r|r||r|r}
&\tn{weight} & &\tn{weight}\\  \hline\hline
\omega &-1 &\ol{\omega} &+1\\
z_j & j-2 &\zb_j &-(j-2)\\
h_2\hb_2&0 &
\end{array} \n
\ee 

%--------------------------------------------------------------------
\subsection{Formal Killing field}\label{sec:summary3} $\,$
 
\one\noi
[\tb{$\sla(2,\C)[\lambda^{-1},\lambda]$-valued Maurer-Cartan form}]
\begin{align}\label{eq:phi2} %\tag{eq:phi2} 
\phi_+&=\bp \cdot& -\frac{1}{2}\gamma \\ \frac{1}{2} \hb_2   &\cdot \ep\xib,\qquad
\phi_0 =\bp \frac{\im}{2}\rho &\cdot \\ \cdot&-\frac{\im}{2}\rho\ep, \qquad
\phi_- =\bp \cdot& -\frac{1}{2}  h_2 \\  \frac{1}{2}\gamma  &\cdot \ep \xi,    \\
\phi_{\lambda}&:=\lambda\phi_+ + \phi_0 + \lambda^{-1}\phi_-
=\bp  \frac{\im}{2}\rho &-\lambda\frac{1}{2}\gamma \xib-\lambda^{-1}\frac{1}{2}  h_2 \xi \n \\
\lambda \frac{1}{2} \hb_2 \xib+ \lambda^{-1}\frac{1}{2}\gamma \xi  &-\frac{\im}{2}\rho\ep.\n 
\end{align}
$$\ed\phi_{\lambda} +\phi_{\lambda}\w\phi_{\lambda} =0.$$

\two \noi
[\tb{$\sla(2,\C)[[\lambda]]\lambda$-valued formal Killing field}]
\be\label{eq:tbY}%\tag{eq:tbY}
\tb{Y}:=\bp -\im\tb{a}&2\tb{c}\\ 2\tb{b}&\im\tb{a} \ep,  
\ee
\be
\tb{a}=\sum_{n=0}^{\infty}\lambda^{2n}a^{2n+1},\qquad
\tb{b}=\sum_{n=0}^{\infty}\lambda^{2n+1}b^{2n+2},\qquad
\tb{c}=\sum_{n=0}^{\infty}\lambda^{2n+1}c^{2n+2}.\n
\ee
\be\label{eq:YKF}
\ed\tb{Y}+[\phi_{\lambda},\tb{Y}]=0.
\ee
\begin{align}\label{eq:b2c2}
&a^1=0, \quad b^2 =-\im\gamma h_2^{-\frac{1}{2}}, \quad c^2=\im h_2^{\frac{1}{2}}, \\
&\{ a^{2n+1},  h_2^{ \frac{1}{2}}b^{2n+2}, h_2^{-\frac{1}{2}}c^{2n+2} \}_{n\geq 0}\subset\mcr.\n
\end{align}
\be\label{eq:b2c2Y}
\det(\bY) =-4b^2 c^2\lambda^2=-4\gamma\lambda^2. 
\ee

\two \noi
[\tb{$\so(4,\C)[\lambda^{-1},\lambda]$-valued Maurer-Cartan form}]
\begin{align}\label{eq:psi4original}
\psi_{+}&= \frac{1}{2}
\bp \cdot & -\gamma   &- \im \gamma   &\cdot \\
\gamma  &\cdot& \cdot & -\hb_2   \\
\im \gamma    &\cdot &\cdot & \im \hb_2  \\
\cdot & \hb_2  &-\im \hb_2   &\cdot \ep\xib, \;\;\;%--------------
\psi_{0}=
\bp \cdot &\cdot & \cdot &\cdot \\
\cdot &\cdot& \rho &\cdot \\
\cdot  & -\rho &\cdot &\cdot\\
\cdot &\cdot &\cdot &\cdot \ep, \;\;\;%-----------------
\psi_{-} = \frac{1}{2}
\bp \cdot & -\gamma & \im \gamma &\cdot \\
\gamma &\cdot& \cdot & -h_2 \\
-\im \gamma   &\cdot &\cdot & -\im h_2   \\
\cdot & h_2   &\im h_2 &\cdot \ep\xi,\\
\psi_{\lambda}&:=\lambda\psi_{+}+\psi_{0}+\lambda^{-1}\psi_{-}.\n
\end{align}
$$
\ed\psi_{\lambda}+\psi_{\lambda}\w\psi_{\lambda}=0.$$
 
\two \noi
[\tb{$\so(4,\C)[[\lambda]]\lambda$-valued formal Killing field}]
\be\label{eq:tbX}
\tb{X}:=\bp 
\cdot& \im(\tb{c}_2+\tb{b}_4)&-(\tb{c}_2-\tb{b}_4)&- \tb{a}_3\\
-\im(\tb{c}_2+\tb{b}_4)& \cdot &-\im \tb{a}_1&-\im(\tb{b}_2+\tb{c}_4)\\
(\tb{c}_2-\tb{b}_4)& \im \tb{a}_1&\cdot &-(\tb{b}_2-\tb{c}_4)\\
\tb{a}_3& \im(\tb{b}_2+\tb{c}_4)& (\tb{b}_2-\tb{c}_4)&\cdot  \ep,
\ee 
\begin{align} 
\tb{a}_1&=\sum_{n=0}^{\infty}(-1)^n\lambda^{4n+0}a^{4n+1},\qquad
\tb{b}_2  =\sum_{n=0}^{\infty}(-1)^n\lambda^{4n+1}b^{4n+2},\qquad
\tb{c}_2  =\sum_{n=0}^{\infty}(-1)^n\lambda^{4n+1}c^{4n+2},\n \\
\tb{a}_3&=\sum_{n=0}^{\infty}(-1)^n\lambda^{4n+2}a^{4n+3},\qquad
\tb{b}_4  =\sum_{n=0}^{\infty}(-1)^n\lambda^{4n+3}b^{4n+4},\qquad
\tb{c}_4  =\sum_{n=0}^{\infty}(-1)^n\lambda^{4n+3}c^{4n+4}.\n  
\end{align} 
\be\label{eq:XKF}
\ed\tb{X}+[\psi_{\lambda},\tb{X}]=0.
\ee

\two\noi
[\tb{Recursive structure equation for the coefficients of  $\bY$}]
\begin{align}\label{eq:abcstrt}%\tag{eq:abcstrt}
\ed  a^{2n+1}&=  (\im \gamma c^{2n+2}+\im h_2 b^{2n+2})\xi + (\im \gamma  b^{2n}+\im \hb_{2} c^{2n})\xib, \\
\ed  b^{2n+2} - \im  b^{2n+2} \rho&= \frac{\im \gamma }{2} a^{2n+3} \xi + \frac{\im}{2}  \hb_{2}a^{2n+1} \xib ,\n\\
\ed  c^{2n+2} +\im c^{2n+2}  \rho&= \frac{\im}{2}  h_2 a^{2n+3}\xi  + \frac{ \im \gamma }{2} a^{2n+1} \xib.\n
\end{align}
  
%---------------------------------------------------------------------
\sub{Jacobi fields and conservation laws}\label{sec:summary4}$\,$

\two\noi
[\tb{Jacobi fields}]

Each coefficient $a^{2n+1}$ is a Jacobi field which lies in the kernel of the Jacobi operator,
\be\label{eq:JBoperator}
\mce:=\delx\delxb+\frac{1}{2}(\gamma^2+h_2\hb_2).
\ee
Here  $\delx,\delxb$ are the covariant derivative operators in the $\xi, \xib$ directions 
(mod $\iinfh$) respectively.
Jacobi fields are the generating functions of the generalized  symmetries of $(\xinfh,\iinfh).$

The set $\{\, a^{2n+1}, \ab^{2n+1}\, \}_{n\geq 1}$ spans the space of higher-order Jacobi fields.
The corresponding higher-order symmetries commute with each other.

\two\noi
[\tb{Conservation laws}]

Set
\be\label{eq:varphin}
\varphi^n:=c^{2n+2}\xi+b^{2n}\xib,\quad n\geq 0.
\ee
Then $$\ed\varphi^n=0,$$ 
and each $\varphi^n$ represents a nontrivial conservation law.

The set $\{\,  [\varphi^n], [\ol{\varphi}^n] \, \}_{n\geq 0}$ spans the space of higher-order conservation laws.

%---------------------------------------------------------------------
%---------------------------------------------------------------------
\sub{Formal moduli spaces}\label{sec:summarymoduli} 
The differential ideal on each of the infinite prolongation spaces $\xinf, \finf, \xinfh, \finfh$ is formally Frobenius. Denote the formal moduli spaces of the integral foliation respectively by,
\begin{align}\label{eq:formalmoduli}
\widehat{\mcm}_{\mcf}&:=\finfh/  (\iinfh)^{\perp}, \qquad  \widehat{\mcm}:=\xinfh/  (\iinfh)^{\perp}, \\
\mcm_{\mcf}&:=\finf/(\iinf)^{\perp}, \qquad   \mcm:=\xinf/(\iinf)^{\perp}.\n
\end{align}
They fit into the commutative diagram:
\be\label{eq:diagramssm}
\qquad\quad
\xymatrix{ \tn{$\widehat{\mcm}_{\mcf}$} \ar[d] \ar[r] & \tn{$\widehat{\mcm}$} \ar[d]\\
\tn{$\mcm_{\mcf}$} \ar[r] & \mcm .
}\ee
   
Any one of  $\widehat{\mcm}_{\mcf}, \mcm_{\mcf}, \widehat{\mcm}, \mcm$ 
will be used as the formal moduli space of CMC surfaces as convenient.

\section{Complexified CMC surfaces}\label{sec:gCMC}
The purpose of this paper is to extend the CMC system 
 to the  CMC hierarchy of evolution equations 
by the higher-order  commuting symmetries generated by $\bY$.
It turns out that the deformations   induced by the CMC hierarchy
do not preserve exactly the class of CMC surfaces.
The structure equation shows that it is necessary to generalize and consider 
a class of \emph{complexified} CMC surfaces for the phase space of the CMC hierarchy.
 
To this end, we give in this section  a precise definition of the complexified 
CMC surfaces. It will be on the formal moduli space of such generalized CMC surfaces
that the CMC hierarchy will be realized as  the hierarchy of commuting flows.
 
%-----------------------------------------------------------------------
%-----------------------------------------------------------------------
\sub{Curvature of a $(1,1)$-form on a Riemann surface}
We first record a preliminary analysis on the curvature associated with 
a nowhere zero $(1,1)$-form on a Riemann surface.

\two
Let $\Sigma$ be a Riemann surface. 
Let $\Omega^{p,q}\to\Sigma$ be the bundle of $(p,q)$-forms.
Let $K=\Omega^{1,0}$ denote the canonical line bundle. 

Suppose $\Upsilon\in H^0(\Sigma, \Omega^{1,1})$ be a nowhere zero $(1,1)$-form.
At each point of $\Sigma$,
there exists a pair of $(1,0)$-form $\xi$ and $(0,1)$-form $\xib$ such that
$$\Upsilon=\frac{\im}{2}\xi\w\xib$$
(note that  $\xib$ is a notation for a $(0,1)$-form 
and does not necessarily mean the complex conjugate $\ol{(\xi)}$ ).
Such a pair of 1-forms $(\xi,\xib)$ is defined up to scaling by  
$$ (\xi,\xib) \to (s \xi, s^{-1}\xib),\quad s\in\C^*.
$$

Let $\pi:F_{\Upsilon}\to\Sigma$ be the associated principal $\C^*$-bundle.\footnotemark\footnotetext{Here $\C^*$ is considered as a multiplicative group.}
From the general theory of G-structures, \cite{Gardner1989},
let $(\xi, \xib)$ be  the tautological pair of 1-forms on $F_{\Upsilon}\to\Sigma$ 
(which we denote by the same notations) such that
$$\pi^*\Upsilon=\frac{\im}{2}\xi\w\xib.
$$
A standard equivalence method argument shows that 
there exists a unique torsion-free (complex) connection 1-form $\rho$ on $F_{\Upsilon}$ 
such that
\begin{align}
\ed\xi&= \;\;\im\rho\w\xi,\n\\
\ed\xib&=-\im\rho\w\xib.\n
\end{align}
The (scalar) curvature $R_{\Upsilon}$ of the $(1,1)$-form $\Upsilon$ is then defined by the equation
$$\ed\rho=R_{\Upsilon} \Upsilon.
$$

%%%%%%%%%%%%%%%
%%%%%%%%%%%%%%%%%%%%%%%%%
\sub{Complexified CMC surfaces}
With this preparation, we give a definition of the  complexified CMC surfaces.
\begin{defn}\label{defn:gCMC}%------------------------%
Let $\gamma^2\in\R$ be a given real structural constant.
A \tb{complexified} CMC surface 
consists of the triple of data $(\Sigma, \Upsilon, \Phi)$,
where $\Sigma$ is a Riemann surface,  
$\Upsilon\in H^0(\Sigma, \Omega^{1,1} )$ is a nowhere zero $(1,1)$-form, 
and $\Phi\in H^0(\Sigma, K^2 )$ is a holomorphic quadratic differential.
They must satisfy the following compatibility condition;
suppose we write (locally)
$$\Upsilon=\frac{\im}{2}\xi\w\xib,$$
for a $(1,0)$-form $\xi$ and a $(0,1)$-form $\xib$.
Let 
\begin{align}
\Phi&=h_2\xi^2,\n\\
\ol{\Phi}&=\hb_2\xib^2,\quad\tn{(complex conjugate of $\Phi$)}.\n
\end{align}
%be the expressions for $\Phi$ and its complex conjugate $\ol{\Phi}$ in terms of $\xi, \ol{\eta}$.
Here $h_2, \hb_2$ are the scalar coefficients
(note again that $\hb_2$ does not necessarily mean the complex conjugate $\ol{(h_2)}$).
Then,
\be\label{eq:gCMCcompat}
R_{\Upsilon}=\gamma^2-h_2 \hb_2.
\ee
Here $R_{\Upsilon}$ is the curvature of the $(1,1)$-form $\Upsilon$.

Let $\mcm^{\C}$ denote the 
\tb{formal moduli space} of the complexified CMC surfaces.
\end{defn}%------------------------%
%It is clear from Eq.\eqref{eq:A} that the generalized CMC surfaces are the correct phase space of  the CMC hierarchy.}

A version of the classical  Bonnet theorem holds
and a complexified CMC surface admits a local embedding into 
a homogeneous space  of $\SO(4,\C)$.
We do not pursue to give the precise description of this space, nor
the related extrinsic geometry of a complexified CMC surface.
 
%%%%%%%%%%%%%%%
%%%%%%%%%%%%%%%%%%%%%%%%%
\sub{Local normal form}
The compatibility equation \eqref{eq:gCMCcompat} can be written in 
the familiar local normal form of (complex) sinh-Gordon equation.

\two
Away from the zero divisor (called ``umbilics") of $\Phi$, 
choose a local holomorphic coordinate $z$ on $\Sigma$
such that $$\Phi=(\ed z)^2.$$
Without loss of generality, let 
\be\label{eq:uzzb}
\xi=e^{u}\ed z,\; \xib=e^u\ed\zb
\ee
for a (complex) scalar function $u=u(z,\zb)$ such that
$$\Upsilon =e^{2u}\frac{\im}{2}\ed z\w\ed\zb.$$
By definition, we have
$$
h_2=\hb_2=e^{-2u}.$$

Differentiate \eqref{eq:uzzb}, and  the corresponding section of 
the torsion-free connection 1-form $\rho$ is given by
$$\rho=\im(u_z\ed z -u_{\zb}\ed \zb).$$
Here $u_z, u_{\zb}$ denote the partial derivatives, etc. 
Differentiate the given $\rho$ again, and the curvature $R_{\Upsilon}$ is given by
$$R_{\Upsilon}=-4e^{-2u}u_{z\zb}.$$
Eq.\eqref{eq:gCMCcompat} is now reduced to the   sinh-Gordon equation,
$$u_{z\zb}+\frac{1}{4}(\gamma^2 e^{2u}-e^{-2u})=0.
$$
%%%%%%%%%%%%
%%%%%%%%%%%%%%%%%%%%%%%%%%%
\sub{Real involution}
Let $\mcm$ be the formal moduli space of ordinary  CMC surfaces.
In the analysis above, note that $u$ is real whenever $\xib=\ol{(\xi)}$ (complex conjugate).
We elaborate on this observation and give a geometric description of 
how $\mcm$ sits inside $\mcm^{\C}$.
  
\two
Let $(\Sigma,\Upsilon,\Phi)$ be a complexified CMC surface.
Consider the associated triple 
$$ (\Sigma,\ol{\Upsilon},\Phi).$$
From the definition, it is easily checked 
that the compatibility equation for this  triple is given by
(following the notations above)
\be\label{eq:RUpb} 
R_{\ol{\Upsilon}}=\gamma^2-\ol{(\hb_2)(h_2)}, \qquad (\tn{complex conjugation}).
\ee

On the other hand, by definition of the curvature of a $(1,1)$-form,  
$$R_{\ol{\Upsilon}}=\ol{R}_{\Upsilon}.
$$
Since $\gamma^2$ is real,
this implies that \eqref{eq:RUpb} holds 
and $(\Sigma,\ol{\Upsilon},\Phi)$ is also a complexified CMC surface.

As a result, the map 
\be\label{eq:tripple}
(\Sigma,\Upsilon,\Phi)\longmapsto (\Sigma,\ol{\Upsilon},\Phi)
\ee
defines an involution on  $\mcm^{\C}$;
$$\mathfrak{i}:\mcm^{\C}\to\mcm^{\C},\quad \mathfrak{i}^2=1_{\mcm^{\C}}.
$$
The fixed point loci of  the involution, i.e., the complexified CMC surfaces 
with the real $(1,1)$-form $\Upsilon=\ol{\Upsilon}$, then exactly correspond to 
the ordinary CMC surfaces.
\begin{prop}\label{prop:gCMC}
Let $\mcm$ be the formal moduli space of CMC surfaces, and
let $\mcm^{\C}$ be the formal moduli space of complexified CMC surfaces.  
There exists an involution $\mathfrak{i}:\mcm^{\C}\to\mcm^{\C}$ defined by \eqref{eq:tripple}
such that $\mcm=(\mcm^{\C})^{\mathfrak{i}}$ is the fixed point loci of \,$\mathfrak{i}$. 
In this sense,  $\mcm^{\C}$ is the complexification of $\mcm.$
\end{prop}

\iffalse
\begin{cor}
The CMC hierarchy defines an infinite sequence of commuting flows on $\mcm^{\C}.$
\end{enumerate}
\end{cor}
\fi

%%%%%%%%%%%%%%%
%%%%%%%%%%%%%%%%%%%%%%%%%
\sub{Remarks}
Let us make a few relevant remarks.
\benu[\qquad a)]
\item
Most of the  results of \cite{Wang2013} summarized in \S\ref{sec:summary},
including infinite prolongation, structure equations, formal Killing field,
Jacobi fields,  and conservation laws,  have their  obvious 
analogues for the complexified CMC surfaces.
One only needs to re-consider  the complex conjugation notation  (overline)
as the formal complex conjugation,   \S\ref{sec:formalconj}.
We leave the rest of  details of the necessary changes for the transition
from the CMC surfaces to the complexified CMC surfaces.

For simplicity, we use the same notations for the corresponding objects for the complexified CMC surfaces.
\item
For a complexified CMC surface, we generally have (following the notations above)
\be 
\quad\qquad \xib\ne\ol{(\xi)},\quad
 \rho\ne\ol{(\rho)},\quad
 \hb_2\ne\ol{(h_2)},\quad (\tn{complex conjugate)}   \n
\ee
and hence
$$\hb_j\ne\ol{(h_j)},\quad j\geq 2.$$
It follows that, on the infinite prolongation space 
$\finfh$ for the complexified CMC surfaces, 
the sequences of functions $\{ h_j \}$ and $\{\hb_j \}$ 
are independent.

We  call $\{ h_j, z_k\}$ and $\{\hb_j, \zb_k\}$
the functions of type $(1,0)$ and $(0,1)$ respectively.
  
\item
For example, there exist two canonical formal Killing fields for the generalized CMC surfaces,
$\bY$ of type $(1,0)$, and $-\ol{\bY}^t$ of type $(0,1)$.
They satisfy the structure equations,
\begin{align}\label{eq:YYbt}
\ed\bY+[\phi_{\lambda},\bY]&= 0,\\
\ed (-\ol{\bY}^t) +[\phi_{\lambda}, (-\ol{\bY}^t)]&= 0.\n
\end{align}
Here $\ol{\bY}$ should be understood as the formal complex conjugation  of $\bY$, 
\S\ref{sec:formalconj}.  For the second equation, note the formal symmetry,
$\phi_{\lambda}= -\ol{\phi_{\lambda}}^t.$
\enu

With these being understood,
a CMC surface would mean a complexified CMC surface from now on.

\section{Infinitely prolonged Gau$\ss$  map}\label{sec:YGauss}
Consider the Lie algebra decomposition,
$$\so(4,\C)=\sla(2,\C)\oplus\sla(2,\C).$$
One of the simplifying factors in the study of the CMC system is that, due to this decomposition, the entire analysis can be based on the simpler Lie algebra $\sla(2,\C)$. We utilize this  and  formulate the CMC hierarchy in terms of a twisted loop algebra $\bmg\subset\sla(2,\C)((\lambda)),$ \eqref{eq:twisted}.
%This will be translated into the original $\so(4,\C)$-setting   in \S\ref{sec:CMChierarchy4}.

In this section,
we give a description of the $\bm{\g}$-valued formal Killing field $\bY$ 
as a part of an infinitely prolonged version of Gau\ss\, map.
This interpretation will play a role in connecting  the CMC system 
to the AKS bi-Hamiltonian hierarchies on 
$ (-\ol{\bmg}^t, \bmg)$.
%The  CMC hierarchy will be obtained by flowing $(-\ol{\bY}^t, \bY)$
%under the AKS hierarchies.% on $-\ol{\bm{\g}_{\geq 1}}^t\times\bm{\g}_{\geq 1}.$

%----------------------------------------------------------------------
%----------------------------------------------------------------------
\sub{Twisted loop algebra} 
Let  $\g=\sla(2,\C).$  
Let
$$\g((\lambda)):=\{\tn{$\g$-valued formal Laurent series in $\lambda$}\}.
$$
Here $\lambda\in\C^*$ is the spectral parameter. 
Define the twisted loop algebra,
\begin{align} 
\bm{\g}&:=\left\{   \, h(\lambda)\in \g((\lambda))\;\vert\; \tn{$h^1_1(\lambda)=-h^2_2(\lambda)$ is even in $\lambda$; $h^1_2(\lambda), h^2_1(\lambda)$ are odd in $\lambda$}  \,   \right\} \label{eq:twisted}\\
&\;\subset\g((\lambda)).\n
\end{align}
Here $h^i_j(\lambda)$'s denote the components of $h(\lambda)$.

The formal Killing field $\bY$ for the CMC system, 
and the extended Maurer-Cartan form for the CMC hierarchy, etc, 
will either take values in $\bm{\g}$,
or at least satisfy the twistedness condition given in \eqref{eq:twisted}. 
%---------------------------------------------------------
\subb{Vector space decomposition}%%%%%%%%%
Consider the vector space decomposition of $\bm{\g}$ into the subalgebras,
\be\begin{array}{rll}\label{eq:vsdecomp}
\bm{\g}&=\;\bm{\g}_{\leq -1} &+^{vs}\quad \bm{\g}_{\geq 0} \\
&\subset\;\g[\lambda^{-1}]\lambda^{-1}&+^{vs}\quad\g[[\lambda]]. 
\end{array}\ee
\noi
Here ``$+^{vs}$" means the direct sum as a vector space and not as a Lie algebra.
The notation  $\bm{\g}_{\leq -1}$ means the subalgebra of polynomial loops 
of $\lambda$-degree $\leq -1$, 
and $\bm{\g}_{\geq 0}$ similarly means the subalgebra of formal power series loops 
of $\lambda$-degree $\geq 0$.
%---------------------------------------------------------
\subb{Dual decomposition} 
We adopt the standard invariant inner product on $\bm{\g}$ defined by
\be\label{eq:inner}
\langle Y_1, Y_2\rangle:=\tn{Res}_{\lambda=0}\big(\tr(Y_1Y_2) \big),
\qquad Y_1, Y_2\in\bmg.
\ee
Here $\tn{Res}_{\lambda=0}$ is the residue operator  which takes the terms of $\lambda$-degree 0.
The corresponding   decomposition of $\bm{\g}=\bm{\g}^*$ dual to \eqref{eq:vsdecomp} is given by
\be\begin{array}{rll}\label{eq:vsdecompdual}
\bm{\g}
& =\;\bm{\g}_{\geq  1} &+^{vs}\quad \bm{\g}_{\leq 0}\\
&\subset\;\g[[\lambda]]\lambda &+^{vs}\quad\g[\lambda^{-1}].  \\
\end{array}\ee
\noi
Here $\bm{\g}_{\geq  1}$, $\bm{\g}_{\leq 0}$ denote the subalgebras which are defined similarly as above
according to their $\lambda$-degrees.
%---------------------------------------------------------
\subb{Determinantal subvarieties in $\pmb{\g}_{\geq 1}$} 
Recall $\C((\lambda))$ is the space of formal Laurent series in $\lambda$.
By definition, the product map $\C((\lambda))\times\C((\lambda))\to\C(\lambda))$ is well defined.
This implies that the determinant function 
$$\det:\bm{\g} \longmapsto\C((\lambda^2))$$
is also well defined. 

Recall the identity, \eqref{eq:b2c2Y},
\be\label{eq:Ydet}
\det(\bY)=-4\gamma\lambda^2\in\C[[\lambda^2]].
\ee
Set $P_{4\gamma\lambda^2}:\bm{\g} \mapsto\C((\lambda^2))$ 
be the defining function for $\bY$,
\be\label{eq:eqYP}
P_{4\gamma\lambda^2}(Y):=\det(Y)+4\gamma\lambda^2,\quad\tn{for $Y\in\bm{\g}$}.\n
\ee
Let $\tn{Y}_{ {4\gamma\lambda^2}}\subset\bm{\g}$ be the corresponding subvariety,
\be\label{eq:varietyYP}
\tn{Y}_{ {4\gamma\lambda^2}}:=\{Y\in\bm{\g}\;\vert\; P_{4\gamma\lambda^2}(Y)=0\}.
\ee

We record the following elementary property of 
the relevant subset $\tn{Y}_{ {4\gamma\lambda^2}}\cap \bm{\g}_{\geq 1}$ without proof.
\begin{lem}\label{lem:lemmaYP}
 Let $\tb{G}_{\geq 0}$ be the formal loop group with Lie algebra $\bmg_{\geq 0}$.
Then, under the adjoint action, $\tb{G}_{\geq 0}$ acts transitively on 
$\tn{Y}_{ {4\gamma\lambda^2}}\cap \bm{\g}_{\geq 1}$.
\end{lem}
%---------------------------------------------------------
\subb{Infinite sequence of quadratic constraints}
From Lem.\ref{lem:lemmaYP},
consider the restriction of the function $P_{4\gamma\lambda^2}$ to $\bmg_{\geq 1}$.
When expanded as a formal power series in $\lambda^2$, 
it gives rise to an infinite sequence of quadratic constraints
for the subvariety  $\tn{Y}_{ {4\gamma\lambda^2}}\cap\bm{\g}_{\geq 1}.$ 
%In the previous work \cite{Wang2013}, the differential algebraic inductive formula for   $\bY$ was obtained by imposing these constraints.
This sequence of quadratic  functions will serve as the commuting Hamiltonians for the CMC hierarchy,
\S\ref{sec:AKS}.

%---------------------------------------------------------------
%---------------------------------------------------------------
\subsection{Zariski open sets}\label{sec:Zariski}
Recall the commutative diagram from \eqref{eq:diagramss},\two\\
\centerline{\xymatrix{\tn{$\Fh{\infty}$}\ar[d] \ar[r] &\; \;\tn{$\hat{X}$}^{(\infty)}\ar[d]  \\ 
 \mcf \ar[r]^{\SO(2)}  & X.  }}\\ 

\two\noi
Define the  Zariski open sets,
\begin{align}\label{eq:opensetsF}
\finfh_*&:=\finfh\setminus\bigcup\{ h_2  = 0, \infty\}\cup\{ \hb_2 = 0, \infty\},\\
\finfh_{**}&:=\finfh_*\setminus\bigcup_{j\geq 3}\{ h_j =\infty \}\cup\{ \hb_j =\infty \}.\n
\end{align}
The corresponding open subsets of $\xinfh$ are denoted by 
\be\label{eq:opensetsX}
\xinfh_*, \xinfh_{**} \subset \xinfh.
\ee

Recall 
$$\{ a^{2n+1},  h_2^{ \frac{1}{2}}b^{2n+2}, h_2^{-\frac{1}{2}}c^{2n+2} \}_{n\geq 0}\subset\mcr=
\C[z_3, z_4, \,... \, ].$$
It follows that the formal Killing fields 
$-\ol{\bY}^t, \bY$ are well defined and smooth on $\finfh_{**}.$
They can be considered as a map
$$(-\ol{\bY}^t, \bY) :\finfh_{**} \longmapsto  -\ol{\bmg_{\geq 1}}^t\times\bmg_{\geq 1}.$$ 
%---------------------------------------------------------
%---------------------------------------------------------

\iffalse
%---------------------------------------------------------
%---------------------------------------------------------
\subsection{Formal Killing field}
Recall the $\bm{\g}_{\geq  1}$-valued formal Killing field $\bY$, \eqref{eq:tbY}.
By definition, it can be considered as a map
$$\tb{Y}:\finfh \longmapsto  -\ol{\bmg_{\geq 1}}^t\times\bmg_{\geq 1}.$$ 
In fact, $\bY$ is defined on a Zariski open subset  $\finfh_{**}\subset\finfh$,  \S\ref{sec:Zariski}.
 \fi
%---------------------------------------------------------
%---------------------------------------------------------
%%%%%%%%%%%%%%%%%%%%%%%%%%%%%%%

%---------------------------------------------------------
%--------------------------------------------------------- 
\sub{Infinitely prolonged Gau$\ss$\, map}
With this preparation, we give a definition of the infinitely prolonged version of Gau\ss\, map.

\two
Consider the diagram in Fig.\ref{fig:InfGauss2}. 
\begin{figure}[h]
\centerline{\xymatrix@R=3.5pc@C=3.5pc{ \tn{$\Fh{\infty}_{**}$}\;\ar[d]_{\pi} 
\;\ar[r]^(.27){ (-\ol{\bY}^t, \bY) } & \quad 
\Big(- \ol{\tn{Y}_{{4\gamma\lambda^2}}}^t\times \tn{Y}_{ {4\gamma\lambda^2}} \Big)
\cap  
\Big(-\ol{\bmg_{\geq 1}}^t\times\bmg_{\geq 1}\Big)\\
\tn{$\mcf$}  & }} 
\caption{Infinitely prolonged Gau\ss\, map}\label{fig:InfGauss2}
\end{figure}
\begin{defn}\label{defn:InfGauss2}
Let $\finfh_{**}\subset\Fh{\infty}$ be the Zariski open set  defined in \eqref{eq:opensetsF}.
The \tb{infinitely prolonged Gau$\ss$\, map} is 
defined by
$$(-\ol{\bY}^t,\bY):\Fh{\infty}_{**}\longmapsto  
\Big(- \ol{\tn{Y}_{{4\gamma\lambda^2}}}^t\times \tn{Y}_{ {4\gamma\lambda^2}} \Big)
\cap  
\Big(-\ol{\bmg_{\geq 1}}^t\times\bmg_{\geq 1}\Big).
$$
Each component $-\ol{\bY}^t,\bY$ satisfies the Killing field equation with respect to $\phi_{\lambda}$,
\begin{align}
\ed\bY+[\phi_{\lambda},\bY]&= 0,\n\\
\ed(-\ol{\bY}^t) +[\phi_{\lambda},(-\ol{\bY}^t)]&= 0.\n
\end{align}
\end{defn}

\one
A relevant observation is that, from the  construction, 
the horizontal map $(-\ol{\bY}^t,\bY)$ in Fig.\ref{fig:InfGauss2} 
is an isomorphism when  restricted to a fiber  of the projection $\pi$, 
which can be considered as the higher-order jet space of the CMC system.
Then  Lem.~\ref{lem:lemmaYP}  
implies that $\Fh{\infty}_{**}$ %of the infinite jet space of the CMC system
is  homogeneous under the combined action of the (formal) Lie groups
$$(\tn{Iso($M$)}^{(\infty)}, \,\ol{\bG_{\geq 0}}^t, \bG_{\geq 0}).
$$
Here $\tn{Iso($M$)}^{(\infty)}$ denotes the infinitely prolonged representation
of $\tn{Iso($M$)}$ in  $\tn{Diffeo}(\Fh{\infty}_{**})$.
%It is clear that the actions by %$\tn{Iso($M$)}^{(\infty)}$,
%$\ol{\bG_{\geq 1}}^t$ (which acts on the $(0,1)$-part)
%and  $\bG_{\geq 1}$ (which acts on the $(1,0)$-part) commute.
\begin{cor}\label{cor:quasi}
The infinite jet space  $\Fh{\infty}$ of the CMC system  is a quasi-homogeneous variety
under the combined action of the Lie groups
$(\tn{Iso($M$)}^{(\infty)}, \,\ol{\bG_{\geq 0}}^t, \bG_{\geq 0}).$
\end{cor}
The homogeneity of the infinite jet space  of  Drinfeld-Sokolov  hierarchies
 is originally due to Frenkel, \cite{Frenkel1995,Frenkel1998}. 

%\marg{In hindsight, the corollary points at the possibility...}

%%%%%%%%%%%% 
%%%%%%%%%%%% 
%%%%%%%%%%%% 
%%%%%%%%%%%% 
%%%%%%%%%%%%%%%%%%%%%  
\section{Adler-Kostant-Symes bi-Hamiltonian hierarchy}\label{sec:AKS}
In this section, we digress from the CMC  system 
and give a description of the Adler-Kostant-Symes (AKS) bi-Hamiltonian hierarchy on $\bmg$ 
via \nR-matrix approach. 
For the general reference on R-matrices and  AKS hierarchies, 
we refer to \cite{Semenov2008}.

%%%%%%%%%%%% 
%%%%%%%%%%%%%%%%%%%%%   
\subsection{\tn{R}-matrix}
Recall the vector space decomposition \eqref{eq:vsdecomp},
\be\begin{array}{rll}\label{eq:vsdecomp9}
\qquad\bm{\g}&=\;\bm{\g}_{\leq -1} &+^{vs}\quad \bm{\g}_{\geq 0}.  
%&\subset\;\g[\lambda^{-1}]\lambda^{-1}&+^{vs}\quad\g[[\lambda]].
\end{array}\ee
Let $\pi_{\leq -1},  \pi_{\geq 0}$ denote the respective projection maps.

Let 
$$\tn{R}:=-\pi_{\leq -1}+\pi_{\geq 0}:\bm{\g}\lra\bmg$$
be the corresponding $\tn{R}$-matrix.  
An R-matrix defines a new Lie bracket
$[\;,\,]_{\tn{R}}$ on $\bmg$ given by
$$[Y_1,Y_2]_{\tn{R}}:=\frac{1}{2}\left([\tn{R}(Y_1),Y_2]+[Y_1,\tn{R}(Y_2)]\right),
\qquad Y_1, Y_2\in\bmg.
$$
In the present case, the new Lie bracket $[\;,\,]_{\tn{R}}$ splits into the direct sum and 
one gets 
$$\bmg_{\tn{R}}=\ominus\bmg_{\leq -1}\oplus\bmg_{\geq 0} \quad\tn{(as a Lie algebra)}.$$
Here $\bmg_{\tn{R}}$ denotes the vector space $\bmg$ equipped with the new Lie bracket 
$[\;,\,]_{\tn{R}}$. The ``$\ominus$" sign indicates that the Lie bracket is given by
$$[(Y_-, Y_+), (Y'_-, Y'_+)]_{\tn{R}}=(-[Y_-,Y'_-], +[Y_+, Y'_+]),$$
for $(Y_-, Y_+), (Y'_-, Y'_+)\in \bmg_{\leq -1}+^{vs}\bmg_{\geq 0}.$

%%%%%%%%%%%% 
%%%%%%%%%%%%%%%%%%%%%  
\subsection{Bi-Poisson structures} 
Let 
\begin{align}\label{eq:intert}
\bm{\sigma}_k&: \bmg\lra \bmg, \quad k\in 2\Z, \\
\bm{\sigma}_k(Y)&:=\lambda^k Y, \qquad Y\in\bmg, \n
\end{align}
be the sequence of intertwining operators.\ftmark\fttext{An intertwining operator is an endomorphism of a Lie algebra which commutes with the adjoint action.}  
Let 
$$\nR_k=\nR\circ\bsigma_k$$ 
be the corresponding sequence of $\nR$-matrices.

It is clear that
for any non-trivial finite linear combination $\bsigma=\sum_{k=i_1}^{i_2} c_k\bsigma_k$, we have that $\tn{Ker}(\bsigma)$ is trivial. 
Hence, the sequence of R-matrices $\{ \nR_k \}$ define a family of Lie algebra structures on $\bmg$.
As a consequence, they induce an infinite dimensional linear family of compatible Lie-Poisson structures on $\bmg^*=\bmg.$

For our purpose, the relevant bi-Poisson structures are given by the pairs of $\nR$-matrices
$$ (\nR, \nR_{\pm 2}).$$

%%%%%%%%%%%% 
%%%%%%%%%%%%%%%%%%%%%  
\subsection{AKS bi-Hamiltonian hierarchy} 
Recall the dual decomposition \eqref{eq:vsdecompdual},
\be\begin{array}{rll}\label{eq:vsdecompdual9}
\bm{\g}^*=\bm{\g}
& =\;\bm{\g}_{\geq  1} &+^{vs}\quad \bm{\g}_{\leq 0}.
%&\subset\;\g[[\lambda]]\lambda &+^{vs}\quad\g[\lambda^{-1}].  \\
\end{array}\ee
Here $\bm{\g}_{\geq  1} = (\bmg_{\leq -1})^*$, and the co-adjoint action
of $X_-\in \bmg_{\leq -1}$ on $Y_+\in \bmg_{\geq 1}$ is given by 
$$[X_-, Y_+]_{\geq 1}.$$
The subscript  ``$_{\geq 1}$" denotes the $\bm{\g}_{\geq  1}$-component of the Lie bracket, etc.
  
With this preparation,  consider the following set of coadjoint  invariant Hamiltonian functions on 
$\bmg^*$,
$$\frac{1}{n}\tn{Res}_{\lambda=0}\big(  \lambda^{-m} \tn{tr}(Y^n)\big),
\quad n,m\in\Z.$$
Since $\bmg^*\subset\sla(2,\C)((\lambda))$, 
choose the set of nontrivial and functionally independent ones,
$$H_{m}:=-\left(\frac{1}{2\im}\right)\frac{1}{2}\tn{Res}_{\lambda=0}
\big(  \lambda^{(-2m-2)} \tn{tr}(Y^{2})\big),\quad m\geq 0$$
(the scaling constants are ornamental).
Its differential is given by
$$\ed H_m=-\left(\frac{1}{2\im}\right)\lambda^{-2m-2}Y\in \bmg=(\bmg^*)^*.$$

With respect to the Lie-Poisson structure on $\bmg_{\nR}^*\simeq\bmg_{\nR}$, 
the Hamiltonian equation for $H_m$ is given by (\cite[Theorem 2.5]{Semenov2008}) 
$$\frac{\ed Y}{\ed t_m}=-\tn{ad}^*_{\bmg}\tn{U}_m(Y),\quad m\geq 0,$$
%=-\frac{1}{2\im}[ (\lambda^{-2m-2}Y)_{\leq -1}, Y] 
where $t_m$ is the time variable, and the formula for the element 
$\tn{U}_m =\frac{1}{2}\nR(\ed H_m)\in \bmg$ is  
$$\tn{U}_m  =\frac{1}{2}\left(\frac{-1}{2\im}\right) \left( -(\lambda^{-2m-2}Y)_{\leq -1} +(\lambda^{-2m-2}Y)_{\geq 0}\right).$$

Under the identification $\bmg=\bmg^*$, we get
$$\frac{\ed Y}{\ed t_m}=-[\tn{U}_m,Y]=-[\tn{U}_m+\left(\frac{1}{4\im}\right)\lambda^{-2m-2}Y,Y].$$ 
Hence the Hamiltonian equation becomes
\be\label{eq:tmflow}
\frac{\ed Y}{\ed t_m}=-[(\frac{1}{2\im}\lambda^{-2m-2}Y)_{\leq -1},Y], \quad m\geq 0. 
\ee
 
\two
It is clear that the hierarchy of $t_m$-flows defined by Eq.\eqref{eq:tmflow} 
is bi-Hamiltonian with respect to the bi-Poisson structures $(\nR, \nR_{\pm 2})$. 
As a consequence, we obtain a commuting bi-Hamiltonian hierarchy of evolution equations on 
$\bmg.$
\begin{defn}\label{defn:AKS}
Let $\bmg\subset\sla(2,\C)((\lambda))$ be the twisted loop algebra \eqref{eq:twisted}.
The \tb{AKS hierarchy} on $\bmg$ is the sequence of commuting bi-Hamiltonian system of equations \eqref{eq:tmflow}.  
\end{defn}
Consequently, the resulting AKS hierarchy will involve the "time" variables 
$$\{  t_m\}_{m\geq 0}.$$
 
Since $\bY$ takes values in $\bmg_{\geq 1}$,
it suffices for the construction of the CMC hierarchy
to consider the restriction of the AKS hierarchy to the strictly positive part 
$$\bmg_{\geq 1}=(\bmg_{\leq -1})^*.$$ 

%%%%%%% %%%%%%%
\subsection{Liouville tori}$\,$\newline
There exists an obvious family of Liouville tori for the AKS hierarchy.

\two
Let $P_q:\bmg\mapsto\C((\lambda^2))$ be a function defined by,
$$P_q(Y):=\det(Y)+q,\quad q\in\C((\lambda^2)).$$
Consider the corresponding determinantal variety defined by  
\be 
\tn{Y}_{ q} =\{\, P_q=0 \}  \subset \bmg.\ee
Note by definition that all the Hamiltonians $H_m$ are constant on this subvariety,
and $\tn{Y}_{ q}$ is clearly invariant under the flow \eqref{eq:tmflow}.
As the constant element $q\in\C((\lambda^2))$ varies, the set of subvarieties 
$\{ Y_{ q}\}$ forms an analogue of the Liouville foliation by invariant tori on $\bmg$
for the AKS hierarchy.

\two
Since $\det(\bY)=-4\gamma\lambda^2,$
the relevant subset for our analysis is  
$$\tn{Y}_{ {4\gamma\lambda^2}}\cap \bmg_{\geq 1}.$$
As noted earlier, $\tn{Y}_{ {4\gamma\lambda^2}} \cap \bmg_{\geq 1}$ 
is an adjoint orbit of the formal loop group 
$\bG_{\geq 0}$.
On the other hand, by construction of the AKS hierarchy, 
the trajectories of the CMC hierarchy
lie in the co-adjoint orbits of  ${\bG_{\leq -1}}$  at the same time.
Here $\bG_{\leq -1}$ denotes the loop group corresponding to 
the polynomial loop algebra $\bmg_{\leq -1}$.

%%%%%%%%%%%% 
%%%%%%%%%%%% 
%%%%%%%%%%%% 
%%%%%%%%%%%% 
%%%%%%%%%%%%%%%%%%%%%  
\section{CMC hierarchy}\label{sec:CMChierarchy2}
%\ftmark  
%\fttext{See \S\ref{sec:formalconj} for the definition of the formal complex conjugation.} 
In this section, the preceding analyses are combined 
to yield the structure equations for the CMC hierarchy.
%The compatibility of the resulting set of equations will be checked in \S\ref{sec:proof}.
 
The key observation is that  
the $\tba_0, t_0$-flows of the ($-\ol{\tn{AKS}}^t$, AKS)-hierarchies
on $(-\ol{\bmg_{\geq 1}}^t,\bmg_{\geq 1})$
are tangent to the infinitely prolonged Gau\ss\, map  $(-\ol{\bY}^t, \bY)$ respectively, 
Lem.\ref{lem:critical}. 
The formal symmetry   of the Maurer-Cartan form $\phi_{\lambda}$,
\be\label{eq:phiformalsymm}
\phi_{\lambda}=-\ol{\phi_{\lambda}}^t, 
\ee
then dictates that the proposed CMC hierarchy should be obtained by
attaching the pair of AKS hierarchies via $(-\ol{\bY}^t, \bY)$
to a combined system of equations on $-\ol{\bmg_{\geq 1}}^t\times\bmg_{\geq 1}$.
The original CMC system, which corresponds to the $\tba_0, t_0$-flows, 
serves as the connecting neck for the operation.
The $\g[[\lambda^{-1},\lambda]]$-valued extended Maurer-Cartan form
for the CMC hierarchy is given  by the formulas \eqref{eq:decompY}, \eqref{eq:bphi}. 
 
%--------------------------------------------------------
%--------------------------------------------------------
\sub{Construction plan}\label{sec:diagram}
Consider the diagram in Fig.\ref{fig:InfGauss2}.
%Note here $$-\ol{\bmg_{\geq 1}}^t\subset\bmg_{\leq -1}.$$
Based on this, the construction plan for the CMC hierarchy can be summarized  
by the following diagram, Fig.\ref{fig:anatomy2}.
\begin{figure}[h]
\begin{align*}
  \tn{CMC hierarchy}&=
-\ol{\tn{mKdV hierarchy}}^t  \oplus\; \tn{CMC system} \;\oplus   \tn{mKdV hierarchy} \\  %& \\
  &\qquad\quad \overset{(-\ol{\bY}^t, \bY)  }{ \xrightarrow{\hspace*{1.5cm}} }   \quad
  -\ol{\tn{AKS hierarchy}}^t   \oplus  \tn{AKS hierarchy.}  
\end{align*}  
\caption{CMC hierarchy}\label{fig:anatomy2}
\end{figure}

Here the appearance of  the mKdV hierarchies is explained by fact that 
the  AKS hierarchy on $\bmg_{\geq 1}$ generated by the $t_m$-flows for $m\geq 0$ 
is a matrix representation of the  mKdV hierarchy, \cite{Wilson1985}.

%--------------------------------------------------------
\subb{Formal complex conjugation}\label{sec:formalconj}
A technical remark is in order.
Define the operation of formal complex conjugation ``$\ol{(\quad)}$" by,
\be
\bp    
\lambda^{\pm 1}\\h_j\\ t_m \\ \xi \\ \rho \ep \quad\xrightarrow{\tn{conjugation}} \quad
\bp  
\lambda^{\mp 1}\\ \hb_j\\ \tba_m \\ \xib \\ \rho \ep.
\ee
For example, the notation $-\ol{\bY}^t$ appeared above means the negative transpose of the formal complex conjugate of $\bY$.

\iffalse
We emphasize again that, for a generalized CMC surface, the reality condition
$$\hb_2=\ol{(h_2)}\quad (\tn{actual complex conjugation})$$
implies that it is an ordinary (real) CMC surface, \S\ref{sec:gCMC}.
\fi
\iffalse
Check that the structure constant $\gamma$ is mapped to $\gamma$ under  the formal conjugation.  
This is to guarantee that the structure equation is compatible with the reality condition
such that 
$$\bphi+\ol{\bphi^t}=0,$$
which is satisfied by the genuine (real, and not complexified) CMC surfaces with 
$\hb_2=\ol{(h_2)}$ (complex conjugate).
\fi

%-------------------------------------------------
%-------------------------------------------------
\sub{Connecting neck} 
%According to the formula \eqref{eq:tmflow}, we start by extending the Killing field equation for $\bY$.
%-------------------------------------------------
\subb{Decomposition of $\bY$}
For $m\geq 0$, set
\begin{align}\label{eq:decompY}
\frac{1}{2\im}\lambda^{-2m-2}\bY&=:U_m+U_{(m+1)} \\
&\subset \bmg_{\leq -1}+^{vs}\bmg_{\geq 0} \n
\end{align}
be the decomposition of the scaled canonical formal Killing field into the respective parts.
The $\bmg_{\leq -1}$-part $U_m$ is given explicitly by
\be
U_m =\bp -\im U_m^a &2U_m^c \\ 2U_m^b&\im U_m^a\ep,\n
\ee
where
\begin{align}\label{eq:Uabc}
U_m^a&=\frac{1}{2\im}\sum_{j=0}^m\lambda^{(2j+0)-(2m+2)}a^{2j+1}, \quad\quad
U_m^c =\frac{1}{2\im}\sum_{j=0}^m\lambda^{(2j+1)-(2m+2)}c^{2j+2},\\
U_m^b&=\frac{1}{2\im}\sum_{j=0}^m\lambda^{(2j+1)-(2m+2)}b^{2j+2}.\n
\end{align}
%-------------------------------------------------
\subb{Key lemma}
Recall from \eqref{eq:phi2}, 
$$\phi_{\lambda}=\lambda\phi_++\phi_0+\lambda^{-1}\phi_-.$$
\begin{lem}\label{lem:critical}
Suppose,
\be\label{eq:torelation}
\ed \ol{t}_0 =-\frac{1}{2}\hb_2^{\frac{1}{2}}\xib, \quad \ed t_0 =-\frac{1}{2}h_2^{\frac{1}{2}}\xi.
\ee
Under this relation,
\be\label{eq:critical}
\lambda \phi_+=-\ol{U}^t_0\ed \ol{t}_0,\quad\lambda^{-1}\phi_-=U_0\ed t_0.
\ee
\end{lem}
Note that the consistency of Eqs.\eqref{eq:torelation} imposes 
the following constraints on the proposed CMC hierarchy,
\be\label{eq:consistency}
\ed (\hb_2^{\frac{1}{2}}\xib)=0,\qquad \ed(h_2^{\frac{1}{2}}\xi)=0.
\ee
\begin{cor}
The infinitely prolonged Gau\ss\, map 
$(-\ol{\bY}^t, \bY)$ is tangent to the $\tba_0, t_0$-flows 
of the ($-\ol{\tn{AKS}}^t$, AKS) hierarchies on $(-\ol{\bmg_{\geq 1}}^t,\bmg_{\geq 1})$
respectively.
\end{cor}

%%%%%%%%%%%%%%%%%%%
\sub{Extended Maurer-Cartan form}\label{sec:822}
Motivated by this and the relation \eqref{eq:phiformalsymm}, 
set the $\g[[\lambda^{-1},\lambda]]$-valued extended Maurer Cartan form $\bphi$ by
\begin{align}\label{eq:bphi}
\bphi&:=-\sum_{m=1}^{\infty}\ol{U}^t_m \ed\ol{t}_m+\phi_{\lambda}
+\sum_{m=1}^{\infty} U_m\ed t_m,\\
&\;=-\sum_{m=0}^{\infty}\ol{U}^t_m \ed\ol{t}_m
+\phi_{0}+\sum_{m=0}^{\infty} U_m\ed t_m.\n
\end{align}
Note that $\bphi$ satisfies the twistedness condition given in
\eqref{eq:twisted}. Note also   the formal identity,
\be\label{eq:bphiformalsymm}
\bphi=-\ol{\bphi}^t.
\ee
%%%%%%%%%%%%%%%%%%%%%%%%%%%%%% 
\subb{Extended structure equation}
 The resulting structure equations for the CMC hierarchy are\ftmark\fttext{The set of equations \eqref{eq:extstrt} is sometimes referred to as the central system, \cite{Falqui1998}.} :
\begin{align}\label{eq:extstrt}
&\ed\bY+[\bphi,\bY]=0,\quad \ed(-\ol{\bY}^t) +[ \bphi,(-\ol{\bY}^t)]=0,\\
&\ed\bphi+\bphi\w\bphi=0.\n
\end{align}
We claim that this system of equations is compatible.

%%%%%%%%%%%%%%%
%%%%%%%%%%%%%%%%%%%%%%%%%%%%%
 \subsection{Assembly}\label{sec:Assembly}
In order to check the consistency of the resulting set of structure equations, we wish to extract a subset of generating equations for \eqref{eq:extstrt}.   
In particular, we are interested in the extension (deformation) of the structure equations for the objects
$$\{ \xi, \xib, \rho, h_2, \hb_2 \}.$$
In a sense, these structure equations are the connection for the assembly. 
The rest of the   equations shall be accounted for 
by the extended Killing field equations for $-\ol{\bY}^t, \bY$.
 
\two
In view of the analysis above, and by imposing the condition that 
the  deformations induced by the CMC hierarchy 
are  conformal  and   preserve Hopf differential \ftmark,\fttext{\S\ref{sec:remarkHopf}.}
we propose the following ansatz for the extension of the structure equations for 
$\{ \xi, \xib, \rho, h_2, \hb_2 \}$: 
\be
\tag*{Eq.($\xi$)}\label{Eqbphi}\left\{
\begin{array} {rl}
\ed\xi-\im\rho\w\xi     &=\sum_{m=1}^{\infty} a^{2m+3}\ed t_m\w\xi, \\
\ed\xib+\im\rho\w\xib &=\sum_{m=1}^{\infty} \ab^{2m+3}\ed \tba_m\w\xib, \n \\
\ed\rho &\equiv R\frac{\im}{2}\xi\w\xib  \qquad \mod \ed\bbt, \ed\ol{\bbt},\\ 
%\sum_{m=1}^{\infty}\ed t_m\w\left((\gamma b^{2m+2}+\hb_2 c^{2m+2})\xib\right),  \\
\ed h_2+2\im h_2\rho&=h_3\xi -2 \sum_{m=1}^{\infty}h_2 a^{2m+3}\ed t_m,\\
\ed \hb_2-2\im \hb_2\rho &=\hb_3\xib-2 \sum_{m=1}^{\infty}\hb_2 \ab^{2m+3}\ed \tba_m.
\end{array}\right.\ee

Let us rewrite the extended Killing field equations,
\be\label{EqbY}
\ed\bY+[\bphi,\bY]=0, \qquad
\ed(-\ol{\bY}^t)+[\bphi,(-\ol{\bY}^t)]=0. \tag*{Eq.(\bY)}
\ee
\be\label{Eqbphi2}
\ed\bphi+\bphi\w\bphi=0. \tag*{Eq.($\bphi$)}
\ee

\two
The claim is that,
\benu[\qquad a)] 
\item  \ref{Eqbphi} and \ref{EqbY} imply \ref{Eqbphi2},
\item \ref{Eqbphi} and \ref{EqbY} are compatible, i.e., the identity $\ed^2=0$ is a formal consequence of these equations.
\enu 
The proof is postponed to \S\ref{sec:proof}.
  
%%%%%%%%%%%%%%%
%%%%%%%%%%%%%%%%%%%%%%%%%%%%%%%
\sub{Remarks}
%-------------------------------------------------------------------
\subb{Conformal, and preserving Hopf differential}\label{sec:remarkHopf}
Note  \ref{Eqbphi} implies $\ed (h_2^{\frac{1}{2}}\xi)=0, \ed (\hb_2^{\frac{1}{2}}\xib)=0.$
Hence,  the  deformations induced by the CMC hierarchy are  
\tb{conformal} and  \tb{preserving Hopf differential}.

%-------------------------------------------------------------------
\subb{Well definedness of \;$[\pmb{\phi},\bY], \,\pmb{\phi}\w\pmb{\phi}$}  
Although the extended Maurer-Cartan form $\bphi$ takes values in $\g[[\lambda^{-1},\lambda]]$, note that each coefficient of the 1-forms $\rho, \ed t_m, \ed\tba_m, m\geq 0,$ in $\bphi$ is $\g[\lambda^{-1},\lambda]$-valued. Since the multiplication map
$$\C[\lambda^{-1},\lambda] \times \C((\lambda)) \lra \C((\lambda))$$
is well defined, the structure equations \ref{Eqbphi}, \ref{EqbY}make sense.

%-------------------------------------------------------------------
\subb{Commuting symmetries}\label{sec:commusymm}
%Recall the formal moduli spaces $\widehat{\mcm}_{\mcf}, \widehat{\mcm}$, \eqref{eq:moduli}.
Note that the extended structure equations  \ref{Eqbphi}, \ref{EqbY}  reduce  to 
the original CMC system if we set\ftmark\fttext{Here 
``$\ed\ol{\tb{t}}, \ed \tb{t}\equiv 0$" means ``modulo $\ed \tba_m,\ed t_m, \forall\, m\geq 1$".} 
$$\ed\ol{\tb{t}}, \ed \tb{t}\equiv 0.$$
Hence, the compatibility  implies that 
the CMC hierarchy induces  a pair of commuting hierarchies of formal symmetry vector fields 
$\{\del_{\tba_m}\}_{m=0}^{\infty}, \{\del_{t_m}\}_{m=0}^{\infty}$ 
(formally dual to $\{\ed\tba_m\}_{m=0}^{\infty}, \{\ed t_m\}_{m=0}^{\infty}$)
on the moduli space $\widehat{\mcm}_{\mcf},$ \eqref{eq:formalmoduli}.

\iffalse
%-------------------------------------------------------------------
\subb{Extrinsic symmetries}
The construction of the CMC hierarchy from the induced pair of 
 ($-\ol{\tn{AKS}}^t$, AKS) hierarchies 
shows that these commuting symmetries of the CMC system are extrinsic 
via the  Gau\ss\, map $(-\ol{\bY}^t, \bY)$.
\fi

%-------------------------------------------------------------------
\subb{Differential system}\label{sec:hierarchysystem}
Consider the product space 
$$\finfh_+:=\finfh\times\{\tba_n\}_{n\geq 0}\times\{ t_m\}_{m\geq 0}.
$$
Let $\iinfh_+$ be the differential ideal on  $\finfh_+$
which cut out the equations \ref{Eqbphi}, \ref{EqbY}.

Strictly speaking, the equality signs in \ref{Eqbphi}, \ref{EqbY}, \ref{Eqbphi2}
should be replaced with  ``$\equiv \mod \iinfh_+$".
For simplicity, we omit this. 
The meaning will be clear from the context.

%-------------------------------------------------------------------
\subb{Affine Toda equation}
The preceding analysis shows that the CMC system (or $\sinh$-Gordon equation) arises as the compatibility equation to join a pair of AKS hierarchies. From the Frenkel's work \cite{Frenkel1998}, it is evident that such a characterization exists for the general affine Toda field equations.

\two
Before we proceed to the proof of compatibility,
we translate the structure equations for the CMC hierarchy 
 into the original $\so(4,\C)$-setting.

%%%%%%%%%%%
%%%%%%%%%%%
%%%%%%%%%%%
%%%%%%%%%%%
%%%%%%%%%%%
%-------------------------------------------------
\section{Translation into $\so(4,\C)$-setting} \label{sec:CMChierarchy4}
Recall the  $\so(4,\C)[\lambda^{-1},\lambda]$-valued Maurer-Cartan form $\psi_{\lambda}$ \eqref{eq:psi4original}, and the corresponding $\so(4,\C)[[\lambda]]\lambda$-valued formal Killing field $\tb{X}$, \eqref{eq:tbX}.
In order to define the extension of $\psi_{\lambda}$, 
we introduce the deformation coefficients $V_m$ analogous to $U_m$ for $\bphi$. 

\iffalse
\be\label{eq:tbX4}
\tb{X}=\bp 
\cdot& \im(\tb{c}_2+\tb{b}_4)&-(\tb{c}_2-\tb{b}_4)&- \tb{a}_3\\
-\im(\tb{c}_2+\tb{b}_4)& \cdot &-\im \tb{a}_1&-\im(\tb{b}_2+\tb{c}_4)\\
(\tb{c}_2-\tb{b}_4)& \im \tb{a}_1&\cdot &-(\tb{b}_2-\tb{c}_4)\\
\tb{a}_3& \im(\tb{b}_2+\tb{c}_4)& (\tb{b}_2-\tb{c}_4)&\cdot  \ep,\n
\ee where
\begin{align}
\tb{a}_1&=\sum_{n=0}^{\infty}(-1)^n\lambda^{4n+0}a^{4n+1},\qquad
\tb{b}_2  =\sum_{n=0}^{\infty}(-1)^n\lambda^{4n+1}b^{4n+2},\qquad
\tb{c}_2  =\sum_{n=0}^{\infty}(-1)^n\lambda^{4n+1}c^{4n+2},\n \\
\tb{a}_3&=\sum_{n=0}^{\infty}(-1)^n\lambda^{4n+2}a^{4n+3},\qquad
\tb{b}_4  =\sum_{n=0}^{\infty}(-1)^n\lambda^{4n+3}b^{4n+4},\qquad
\tb{c}_4  =\sum_{n=0}^{\infty}(-1)^n\lambda^{4n+3}c^{4n+4}.\n 
\end{align}\fi

\two
For each $m\geq 1$, define the $\so(4,\C)$-valued function
$V_m$ depending on the pairity of $m$ as follows. Here we set $a^{-1}=0$.

\two
\noi [\tb{case} $m$ is even]\\
\indent Define
$$ \epsilon(m)=\begin{cases} 
+1 &\mbox{if } m \equiv 0 \\ 
-1 & \mbox{if } m \equiv 2  \end{cases} \pmod{4}.$$
Let
\begin{align}\label{eq:Vabce}
V_m^{\tb{a}_1}&=  \epsilon(m)\sum_{j=0}^{\frac{m}{2}} (-1)^{j}\lambda^{(4j-2)-(2m+2)}a^{4j-1},\qquad
V_m^{\tb{a}_3} = \epsilon(m)\sum_{j=0}^{\frac{m}{2}} (-1)^{j}\lambda^{(4j+0)-(2m+2)}a^{4j+1}, \\
V_m^{\tb{b}_2}&= \epsilon(m)\sum_{j=0}^{\frac{m}{2}} (-1)^{j}\lambda^{(4j-1)-(2m+2)}b^{4j+0},\qquad
V_m^{\tb{b}_4} = \epsilon(m)\sum_{j=0}^{\frac{m}{2}} (-1)^{j}\lambda^{(4j+1)-(2m+2)}b^{4j+2},\n\\
V_m^{\tb{c}_2}&= \epsilon(m)\sum_{j=0}^{\frac{m}{2}} (-1)^{j}\lambda^{(4j-1)-(2m+2)}c^{4j+0},\qquad
V_m^{\tb{c}_4} = \epsilon(m)\sum_{j=0}^{\frac{m}{2}} (-1)^{j}\lambda^{(4j+1)-(2m+2)}c^{4j+2}.\n 
\end{align}
 
\two
\noi [\tb{case} $m$ is odd]\\
\indent Define
$$ \epsilon(m)=\begin{cases} 
-1 &\mbox{if } m \equiv 1 \\ 
+1 & \mbox{if } m \equiv 3 \end{cases} \pmod{4}.$$
Let
\begin{align}\label{eq:Vabco}
V_m^{\tb{a}_1}&= \epsilon(m)\sum_{j=1}^{\frac{m+1}{2}} (-1)^{j}\lambda^{(4j-4)-(2m+2)}a^{4j-3},\qquad
V_m^{\tb{a}_3} = \epsilon(m)\sum_{j=1}^{\frac{m+1}{2}} (-1)^{j}\lambda^{(4j-2)-(2m+2)}a^{4j-1},\\
V_m^{\tb{b}_2}&= \epsilon(m)\sum_{j=1}^{\frac{m+1}{2}} (-1)^{j}\lambda^{(4j-3)-(2m+2)}b^{4j-2},\qquad
V_m^{\tb{b}_4} = \epsilon(m)\sum_{j=1}^{\frac{m+1}{2}} (-1)^{j}\lambda^{(4j-1)-(2m+2)}b^{4j+0},\n\\
V_m^{\tb{c}_2}&= \epsilon(m)\sum_{j=1}^{\frac{m+1}{2}} (-1)^{j}\lambda^{(4j-3)-(2m+2)}c^{4j-2},\qquad
V_m^{\tb{c}_4} = \epsilon(m)\sum_{j=1}^{\frac{m+1}{2}} (-1)^{j}\lambda^{(4j-1)-(2m+2)}c^{4j+0}.\n 
\end{align}

\two
Now set
\be
V_m =\bp 
\cdot& \im(V_m^{\tb{c}_2}+V_m^{\tb{b}_4})&-(V_m^{\tb{c}_2}-V_m^{\tb{b}_4})&- V_m^{\tb{a}_3}\\
-\im(V_m^{\tb{c}_2}+V_m^{\tb{b}_4})& \cdot &-\im V_m^{\tb{a}_1}&-\im(V_m^{\tb{b}_2}+V_m^{\tb{c}_4})\\
(V_m^{\tb{c}_2}-V_m^{\tb{b}_4})& \im V_m^{\tb{a}_1}&\cdot &-(V_m^{\tb{b}_2}-V_m^{\tb{c}_4})\\
V_m^{\tb{a}_3}& \im(V_m^{\tb{b}_2}+V_m^{\tb{c}_4})& (V_m^{\tb{b}_2}-V_m^{\tb{c}_4})&\cdot  \ep.
\ee
Define the $\so(4,\C)[[\lambda^{-1},\lambda]]$-valued extended  Maurer-Cartan form $\bpsi$ by
\be\label{eq:extpsi4}
\bpsi:=-\sum_{m=1}^{\infty}\ol{V}^t_m\ed \tba_m+\psi_{\lambda}+\sum_{m=1}^{\infty}V_m\ed t_m.
\ee

Note the formal identity
$$\bpsi=-\ol{\bpsi}^t.$$
The corresponding extended structure equations are:
\begin{align}\label{eq:psiXeq}
\ed\tb{X}+[\bpsi,\tb{X}]&=0,\quad  \ed(-\ol{\tb{X}}^t)+[\bpsi,(-\ol{\tb{X}}^t)]=0, \\
\ed\bpsi+\bpsi\w\bpsi&=0.\n
\end{align}
It can be checked that these equations are equivalent to \eqref{eq:extstrt}.
 
%%%%%%%%%%%
%%%%%%%%%%%
%%%%%%%%%%%
%%%%%%%%%%%
%%%%%%%%%%%
%-------------------------------------------------
\section{Proof of compatibility}\label{sec:proof}
Let us first rewrite the compatibility equation \ref{Eqbphi2}  
 in such a way that is suitable for the computation in this section.

Consider the decomposition
\begin{align}
\bphi&=-\sum_{n=0}^{\infty}\ol{U}^t_n\ed\tba_n+\phi_0+\sum_{m=0}^{\infty} U_m\ed t_m\n\\
&:=\bphi_++\phi_0+\bphi_-.\n
\end{align}
The $\bphi_+$-terms have $\lambda$-degree $\geq 1$, and 
the $\bphi_-$-terms have $\lambda$-degree $\leq -1$.
The  $\phi_0$-term, \eqref{eq:phi2}, has $\lambda$-degree 0.

In terms of this decomposition,  
\ref{Eqbphi},   \ref{EqbY}, \ref{Eqbphi2} can be organized as follows.
\be\tag{A}\label{eq:A}\\
\left\{
\begin{array} {rl}
\ed\xi-\im\rho\w\xi     &=\sum_{m=1}^{\infty} a^{2m+3}\ed t_m\w\xi, \\
\ed\xib+\im\rho\w\xib &=\sum_{m=1}^{\infty} \ab^{2m+3}\ed \tba_m\w\xib, \n \\
\ed\rho &\equiv R\frac{\im}{2}\xi\w\xib   \qquad \mod \ed\bbt, \ed\ol{\bbt}, \\
%\sum_{m=1}^{\infty}\ed t_m\w\left((\gamma b^{2m+2}+\hb_2 c^{2m+2})\xib\right), \\
\ed h_2+2\im h_2\rho&=h_3\xi -2 \sum_{m=1}^{\infty}h_2 a^{2m+3}\ed t_m,\\
\ed \hb_2-2\im \hb_2\rho &=\hb_3\xib-2 \sum_{m=1}^{\infty}\hb_2 \ab^{2m+3}\ed \tba_m.
\end{array}\right.\ee
  
\iffalse
\begin{align}\tag{B}\label{eq:B}
\underbrace{\left(\ed \phi_{\lambda}+ \phi_{\lambda}\w \phi_{\lambda}\right)}_\text{B$^1$}
&+\underbrace{\sum_{m=1}^{\infty}\left(\ed U_m +[ \phi_{\lambda},U_m]\right)\w \ed t_m}_\text{B$^0$}
+\underbrace{\sum_{m=1}^{\infty}\left(\ed \ol{U}_m +[ \phi_{\lambda},\ol{U}_m]\right)\w \ed \tba_m}_\text{$\bar{\nB}^0$}\n\\
&+\underbrace{\sum_{m,\ell=1}^{\infty}\frac{1}{2}[U_m,U_{\ell}] \ed t_m\w\ed t_{\ell} }_\text{B$^{-2}$} 
+\underbrace{\sum_{m,\ell=1}^{\infty}\frac{1}{2}[\ol{U}_m,\ol{U}_{\ell}] \ed \tba_m\w\ed \tba_{\ell} }_\text{B$^{2}$}  \n \\
&+\underbrace{\sum_{m,\ell=1}^{\infty} [U_m,\ol{U}_{\ell}] \ed t_m\w\ed \tba_{\ell} }_\text{B$^{9}$}.\n
\n
\end{align}
\fi

\begin{align} 
& \Big(\ed \bphi_++\bphi_+\w\bphi_+ + [\phi_0,\bphi_+]+[\bphi_+,\bphi_-]_{\oplus} \Big)=0,
 \tag{B$_{\oplus}$}\label{eq:Bp} \\
& \Big(\ed \phi_0 +[\bphi_+,\bphi_-]_{0}\Big)  =0,
\tag{B$_0$}\label{eq:B0} \\
& \Big(\ed \bphi_-+\bphi_-\w\bphi_- + [\phi_0,\bphi_-]+[\bphi_-,\bphi_+]_{\ominus}\Big)=0.\n \tag{B$_{\ominus}$}\label{eq:Bm}
\end{align}

\[\ed\bY+[\bphi,\bY]=0,\quad \ed(-\ol{\bY}^t)+[\bphi,(-\ol{\bY}^t)]=0.\tag{C}\label{eq:C}
\]

\two\noi
Here the equation $\ed\bphi+\bphi\w\bphi=0$
 is decomposed into the three parts
\eqref{eq:Bp},\eqref{eq:B0},\eqref{eq:Bm} according to their $\lambda$-degrees.
The subscripts ``$_{\oplus}$,  $_0$,  $_{\ominus}$" denote 
the terms of $\lambda$-degree $\geq 1$, $=0$, $\leq -1$ respectively.

%%%%%%%%%%%%%%%%%%%%%%%%%%%%%%%%%%%
%\sub{Main theorem}\label{sec:proof2}
\two
We now state the main theorem of this paper.
\begin{thm}\label{thm:main}
The system of equations Eq.\eqref{eq:A},  Eq.\eqref{eq:C} for the CMC hierarchy
is compatible, i.e., $\ed^2=0$ is a formal consequence of the structure equations. 
\end{thm}
 
Note that the compatibility equation of Eq.\eqref{eq:C} is Eqs.\eqref{eq:Bp},\eqref{eq:B0},\eqref{eq:Bm}.
For a proof of  the theorem, 
we first show that
Eqs.\eqref{eq:Bp},\eqref{eq:B0},\eqref{eq:Bm} vanish  
modulo Eq.\eqref{eq:A}, Eq.\eqref{eq:C}.
Then,
we check that
Eq.\eqref{eq:A}  is compatible with
Eqs.\eqref{eq:Bp},\eqref{eq:B0},\eqref{eq:Bm}, and Eq.\eqref{eq:C}.
  
%---------------------------------------------------
\sub{Eqs.\eqref{eq:Bp},\eqref{eq:Bm}}\label{sec:Bpm}
The claim is that
$$\tn{ Eqs.\eqref{eq:Bp},\eqref{eq:Bm} $\equiv 0\mod$ Eq.\eqref{eq:A}, Eq.\eqref{eq:C}}.$$
This follows from the commuting property of the AKS  bi-Hamiltonian hierarchy.
We record a proof for completeness.
%%%%%%%%%%%%%%%%%%
\subb{$\ed t_m\w\ed t_{\ell}, \ed \tba_m\w\ed \tba_{\ell}$-terms}\label{sec:dtmtl}
It is clear that this part of the claim is equivalent to the following lemma 
and its formal complex conjugate.
\begin{lem}\label{lem:tmUtl} 
For all $m,\ell \geq 0$,
\be\label{eq:tmUtl} 
\del_{t_m}U_{\ell}-\del_{t_{\ell}}U_m+[U_m,U_{\ell}]=0.  
\ee
\end{lem}
\begin{proof}
We give a proof  by $\lambda$-degree counting.

\tb{Step} 1.\,
Recall the decomposition
\be\label{eq:YUm}
\bY=2\im\lambda^{2m+2} (U_m+U_{(m+1)}).
\ee
From Eq.\eqref{eq:C}, we have
\be\label{eq:tellY}
\del_{t_{\ell} } \tb{Y}=-[U_{\ell},\tb{Y}].
\ee
Here $\del_{t_{\ell} }=\dd{}{t_{\ell} }$ denotes the partial derivative operator.
%Hence,
%\be%\label{eq:tnYY}
%\del_{t_m} \tb{Y}=-[U_m,\tb{Y}]=-[U_m,2\im\lambda^{2m+2}U_{(m+1)}].\n
%\ee

\one
\tb{Step} 2.\,
Substitute  \eqref{eq:YUm} to \eqref{eq:tellY}, and one gets
\begin{align}  
 \del_{t_{\ell}}U_m+ \del_{t_{\ell}}U_{(m+1)}
&=- [U_{\ell},U_m+U_{(m+1)}]   \label{eq:tlUm}  \\
&=- [U_{\ell},U_m]- [U_{\ell},U_{(m+1)}].\n
\end{align}
Interchange $\ell, m$ and take the difference, and one gets
\begin{align}\label{eq:longeq}
 (\del_{t_{m}}U_{\ell}-\del_{t_{\ell}}U_m)
&+( \del_{t_m}U_{(\ell+1)}- \del_{t_{\ell}}U_{(m+1)}) \\
&+ (2[U_m,U_{\ell}])
+( [U_m,U_{(\ell+1)}]- [U_{\ell},U_{(m+1)}])=0.\n
\end{align}
\begin{lem}\label{lem:Y2}
For all $m, \ell\geq 0$,
\be\label{eq:UUpa}
 [U_m,U_{\ell}]
+ ([U_m,U_{(\ell+1)}]- [U_{\ell},U_{(m+1)}])
+ [U_{(m+1)},U_{(\ell+1)}]=0.
\ee
\end{lem}
\begin{proof} 
This follows from the trivial identity,
$$[\bY,\bY]=0=[ U_m+U_{(m+1)},  U_\ell+U_{(\ell+1)}].
$$
\end{proof}
\tb{Step} 3.\,
Substitute  \eqref{eq:UUpa} to  \eqref{eq:longeq}, and one gets
\begin{align}%\label{eq:neweq}
\Big( (\del_{t_{m}}U_{\ell}&-\del_{t_{\ell}}U_m)+[U_m,U_{\ell}]\Big) \\
&+\Big((\del_{t_m}U_{(\ell+1)}-\del_{t_{\ell}}U_{(m+1)})
-[U_{(m+1)},U_{(\ell+1)}]\Big)=0.\n
\end{align}
In this equation, 
the $\lambda$-degree of the first line is $\leq -1$,
whereas the $\lambda$-degree of the second line is $\geq 0$.
It follows that the equation above holds separately, i.e.,
\begin{align}
&(\del_{t_{m}}U_{\ell}-\del_{t_{\ell}}U_m)+[U_m,U_{\ell}]=0,\n\\
&(\del_{t_m}U_{(\ell+1)}-\del_{t_{\ell}}U_{(m+1)})
-[U_{(m+1)},U_{(\ell+1)}]=0.  \label{eq:secondeq}
\end{align}
This completes the proof of Lem.~\ref{lem:tmUtl}.
\end{proof}
 
%---------------------------------------------------
\subb{$\ed t_m\w\ed \tba_{n}$-terms}
Similarly as above, 
the claim is equivalent to the following lemma.
\begin{lem}\label{lem:tmUtn} 
For all $m, n\geq 0$,
\be%\label{eq:tmUtn}
\del_{t_m}\ol{U}^t_{n}+\del_{\tba_{n}}U_m
+[U_m,\ol{U}^t_{n}]_{\oplus}+[U_m,\ol{U}^t_{n}]_{\ominus}=0.\n
\ee
\end{lem}
\begin{proof}
By Eq.\eqref{eq:C}, we have
$$\del_{\tba_{n}}(U_m+U_{(m+1)})=[\ol{U}^t_{n}, U_m+U_{(m+1)}].
$$
Since the  terms in $U_{(m+1)}, \ol{U}^t_{n}$ are of $\lambda$-degree $\geq 0$,
take the $\ominus$-terms (of $\lambda$-degree $\leq -1$) only and one gets
\be\label{eq:tbnUm}
\del_{\tba_{n}} U_m=[\ol{U}^t_{n}, U_m]_{\ominus}.
\ee
Take the conjugate transpose of this equation and interchange $m, n$, and one gets
$$\del_{t_{m}} \ol{U}^t_{n}=[\ol{U}^t_{n}, U_m]_{\oplus}.$$
\end{proof}

%---------------------------------------------------
\sub{Eq.\eqref{eq:B0}}\label{sec:B0}
This gives the formula for $\ed\rho$ in Eq.\eqref{eq:A}.

In order to show the compatibility,
one needs to verify that $\ed^2\rho=0$ is an identity.
This is equivalent to,
$$\ed\eqref{eq:B0}\equiv 0\mod \tn{ Eqs.\eqref{eq:Bp},\eqref{eq:B0},\eqref{eq:Bm}, Eq.\eqref{eq:C}}.
$$
 
 %---------------------------------------------------
\sub{{$\ed^2\rho$, or $\ed\eqref{eq:B0}$}\label{sec:ddrho}}
From Eq.\eqref{eq:B0}, we have
\be\label{eq:B0drho}
\ed\phi_0+(\bphi_+\w\bphi_-+\bphi_-\w\bphi_+)_0=0.
\ee
Differentiate this equation using the given formulas for $\ed\bphi_{\pm}$.
After collecting terms, one gets
\[
\Big([\bphi_-,\bphi_+\w\bphi_+]+[\bphi_+,\bphi_-\w\bphi_-]
+[\bphi_-,[\bphi_+,\bphi_-]_{\oplus}]+[\bphi_+,[\bphi_+,\bphi_-]_{\ominus}]\Big)_0=0.
\]
Considering the $\lambda$-degrees, this is equivalent to
\[
\Big([\bphi_-,\bphi_+\w\bphi_+]+[\bphi_+,\bphi_-\w\bphi_-]
+[\bphi_-,[\bphi_+,\bphi_-] ]+[\bphi_+,[\bphi_+,\bphi_-] ]\Big)_0=0.
\]
The expression inside the parenthesis vanishes by cancellation.

\iffalse
We mention in passing that a similar computation works to check the compatibility of 
Eqs.\eqref{eq:Bp}, \eqref{eq:Bm}. In the Eq.\eqref{eq:Bm} case,
the calculation of $\ed (\tn{Eq.} \eqref{eq:Bm})=0$ reduces to the identity
\begin{align*}
&\Big(\bphi_+\w\bphi_-\w\bphi_- - \bphi_-\w\bphi_-\w\bphi_+
+\bphi_-\w[\bphi_+,\bphi_-]_{\oplus}-[\bphi_+,\bphi_-]_{\oplus}\w\bphi_-\Big)_{\oplus}\\
&=\Big( [\bphi_+,\bphi_-]_{\ominus}\w\bphi_-   - \bphi_- \w[\bphi_+,\bphi_-]_{\ominus}    \Big)_{\oplus}=0.
\end{align*}
\fi
 
%-------------------------------------------------------------------------------------
\twoline
%-------------------------------------------------------------------------------------

\one
For the remainder of proof,
we first derive Eq.\eqref{eq:A} from
Eqs.\eqref{eq:Bp},\eqref{eq:B0},\eqref{eq:Bm}, Eq.\eqref{eq:C}.
Then, we show that Eq.\eqref{eq:A} is compatible.
%%%%%%%%%%%%%
%%%%%%%%%%%%%%%%%%%%%%%%%%%%%%
\sub{Eq.\eqref{eq:A}}
The analysis thus far shows the compatibility of the $(-\ol{\tn{AKS}}^t,\tn{AKS})$-hierarchy
on $-\ol{\bmg_{\geq 1}}^t\times\bmg_{\geq 1}$,
under the constraints that 
$$\ed \tba_0=- \frac{1}{2}\hb_2^{\frac{1}{2}}\xib,\quad \ed t_0=-\frac{1}{2}h_2^{\frac{1}{2}}\xi.$$

From \eqref{eq:b2c2},
the formula for $\ed h_2$ is included in Eq.\eqref{eq:C}, and hence it is compatible.
Note that Eq.\eqref{eq:A} implies,
$$\ed (\hb_2^{\frac{1}{2}}\xib)=0, \quad \ed (h_2^{\frac{1}{2}}\xi)=0.$$
The formulas for $\ed\xi, \ed\xib$ will follow from these equations.
 
%---------------------------------------------------
\subb{Formula for $\ed h_2$}\label{sec:dh2}
We first derive the formula  for $\ed h_2$.

Recall
$$U_0=\bp \cdot & h_2^{\frac{1}{2}}\\-\gamma h_2^{-\frac{1}{2}}&\cdot\ep \lambda^{-1}.
$$
Apply the formula \eqref{eq:tbnUm} for the case $m=0$,
\[
\del_{\tba_{n}} U_0= [\ol{U}^t_{n}, U_0]_{\ominus}.
\]
Since the terms in $\ol{U}^t_{n}$ have $\lambda$-degree $\geq 1$, 
whereas $U_0$ has $\lambda$-degree  $-1$, 
we have $[\ol{U}^t_{n}, U_0]_{\ominus}=0$.
Hence,
$$\del_{\tba_{n}} U_0=0, \quad\forall \,n\geq 0.$$

On the other hand, collecting the terms of $\lambda$-degree $-1$ 
from \eqref{eq:tlUm} for the case $m=0$, one gets
\[\partial_{t_{\ell}}U_0=-[U_{\ell}, U_{(1)}]_{\ominus 1}.
\]
Here the subscript ``$\ominus 1$" means the terms of $\lambda$-degree $-1$.
Consider the identity
$$[U_{\ell}+U_{(\ell+1)},U_{0}+U_{(1)}]=0,$$
(this is, up to constant scale, the trivial equation $[\bY,\bY]=0$).
Collecting the terms of $\lambda$-degree  $-1$, one gets
\begin{align*}
-[U_{\ell}, U_{(1)}]_{\ominus 1}&=[U_{(\ell+1)}, U_0]_{\ominus 1}\\
&=[(U_{(\ell+1)})_{0}, U_0].
\end{align*}
Here "$(U_{(\ell+1)})_{0}$" denotes the terms of $\lambda$-degree 0 in $U_{(\ell+1)}$.
This gives the desired formula for $\ed h_2$. 
%and hence for $\ed\xi$  from the relation $\ed (h_2^{\frac{1}{2}}\xi)=0.$

%---------------------------------------------------
\subb{Formula for $\ed\xi$ }\label{sec:dxi}
Given the formula for $\ed h_2$, the formula for $\ed\xi$ is determined from the equations,
\begin{align}
\xi&=-2h_2^{-\frac{1}{2}} \ed t_0, \n\\
\ed\xi&=h_2^{-\frac{1}{2}-1}\ed h_2\w\ed t_0 \n\\
&=-\frac{1}{2}h_2^{-1}\ed h_2\w\xi. \n
\end{align}
The compatibility equation $\ed^2\xi=0$   follows from this and the compatibility of $\ed h_2$.
This completes the proof for the compatibility of Eq.\eqref{eq:A}.

\section{Extension of conservation laws}\label{sec:extcvlaw}%joanne
Recall the sequence of higher-order conservation laws  $\varphi^n$,  \eqref{eq:varphin}.
We show that they admit an extension to the conservation laws of the  CMC hierarchy.

Let us  introduce a relevant notation. 
Let $$\mcd:=\mcl_{\lambda \dd{}{\lambda}}$$
be the Euler operator with respect to the spectral parameter $\lambda$.
For a scalar function, or a differential form $A$, the notation  $\bdot{A}$ (upper-dot) would mean
the application of the Euler operator,
$$\bdot{A}=\mcd (A).$$
 
\two
Set 
\be\label{eq:varphibY}
\varphi_{\bY}:=  \tr ( \bY\bdot{\bphi}).
\ee
\begin{thm}\label{thm:extcvlaw}
Consider the $\C[[\lambda^{-2},\lambda^2]]$-valued 1-form $\varphi_{\bY}$, \eqref{eq:varphibY}.
\benu[\qquad a)]
\item The 1-form $\varphi_{\bY}$ is closed,
$$\ed \varphi_{\bY}=0.$$ 
When expanded as a formal series in $\lambda^{-2},\lambda^2$,
each coefficient represents a conservation law of the CMC hierarchy.
\item $\varphi_{\bY}$ represents an extension of the sequence of conservation laws $\varphi^n$
in the following sense;
\be\label{eq:cvlawextend}
\varphi_{\bY}+\im\ed\bba\equiv-2\gamma \sum_{n=0}^{\infty}\lambda^{2n}\varphi^n
\mod \ed\ol{\tn{\bbt}}, \ed\tn{\bbt}.\ee
\enu
\end{thm}
\begin{proof}
a)\; Differentiate $\varphi_{\bY}$, and one gets
\begin{align*}
\ed\varphi_{\bY}&=  \tr \big( \ed\bY\w\bdot{\bphi}+\bY\ed\bdot{\bphi}\big) \\
&= \tr \big( (-\bphi\bY+\bY\bphi)\w\bdot{\bphi}-\bY(\bphi\w\bdot{\bphi}+\bdot{\bphi}\w\bphi)\big)\\
&=0.
\end{align*}
b)\; Modulo $\ed\ol{\tn{\bbt}}, \ed \tn{\bbt}$,
\begin{align} 
\varphi_{\bY}=\tr ( \bY\bdot{\bphi})&\equiv\tr \left[\bp -\im\bba & 2\bbc \\ 2\bbb &  \im\bba\ep 
\bp \cdot &-\lambda\frac{1}{2}\gamma \xib+\lambda^{-1}\frac{1}{2}  h_2 \xi \\
\lambda \frac{1}{2} \hb_2 \xib- \lambda^{-1}\frac{1}{2}\gamma \xi  & \cdot\ep\right]
\mod \ed\ol{\tn{\bbt}}, \ed \tn{\bbt},  \label{eq:varphiYform} \\
&=\bbc(- \lambda^{-1} \gamma \xi+\lambda \hb_2 \xib)
+\bbb(-\lambda \gamma \xib+\lambda^{-1} h_2 \xi) \n\\
&= - \gamma(\lambda^{-1}  \bbc \xi+\lambda\bbb \xib)
+(\lambda^{-1} h_2\bbb\xi+\lambda \hb_2\bbc \xib). \n
\end{align}
On the other hand, we have
\begin{align}
\ed\bba&\equiv\lambda^{-1}(\im \gamma \bbc+\im h_2 \bbb)\xi 
+ \lambda(\im \gamma  \bbb+\im \hb_{2} \bbc)\xib  \mod \ed\ol{\tn{\bbt}}, \ed \tn{\bbt}, \label{eq:dbba}\\
\sum_{n=0}^{\infty}\lambda^{2n}\varphi^n&=\sum_{n=0}^{\infty}
\lambda^{2n}(c^{2n+2}\xi+b^{2n}\xib), \qquad \tn{(here we set $b^0=c^0=0$)} \label{eq:varphiseries}\\
&=\lambda^{-1}\bbc \xi+\lambda \bbb\xib.\n
\end{align}
Eq.\eqref{eq:cvlawextend} follows from \eqref{eq:varphiYform}, \eqref{eq:dbba}, \eqref{eq:varphiseries}.
\end{proof}

%%%%%%%%%%%%%%%%%%%%%%%%%%%%%%%%%%%%%
\iffalse
Recall $\mcr=\C[z_3,z_4,\, ... ].$ 
Set $r=(h_2\hb_2)^{\frac{1}{2}}$, and note that
\be\label{eq:delr}
\del_{t_m}(r)=-a^{2m+3} r.
\ee
Let  $\widetilde{\mcr}$ be  the ring generated by 
$$\widetilde{\mcr}:=\langle r^{\pm 1},\mcr,\ol{\mcr}\rangle.$$
Eq.\eqref{eq:C} and Eq.\eqref{eq:delr} imply that 
$$\del_{t_0}, \del_{\ol{t}_0}, \del_{t_j}, \del_{\tba_j}:\widetilde{\mcr}\to\widetilde{\mcr}.$$
\fi

\iffalse  
\beit
\item
the partial derivative operators 
$\del_{t_m}, \del_{\tba_m}:\widetilde{\mcr}\to \widetilde{\mcr}, m\geq 0,$ 
invert the parity of spectral weight, \eqref{eq:spectralweight}.
\item
the open set $\xinfh_{**}$ is homotopic to $\Xh{1}_{**}=\Xh{1} 
 \setminus \{ h_2=0,\infty\}   \cup \{\hb_2=0,\infty\}$.
\enit
The first claim follows from Eq.\eqref{eq:C}.
The second claim follows from the observation  that
$\xinfh_{**}$ is the projective limit of a sequence of $\PP^{1}\setminus\{\infty\} \simeq \C$-bundles
over $\Xh{1}_{**}$. 

Note  that 
$$\Xh{1}_{**}\to X\simeq \ES^3\times\ES^2$$
is a $\C^*$-bundle, and 
the possible homology of  $\Xh{1}_{**}$ is generated by 
a nontrivial cycle on a fiber $\C^*$.
Note also that such a cycle is an integral curve of $\Ih{1}$.
The relevant observation here is that the 1-forms
\[
 \Big( \del_{t_k}\hat{c}^{2n+2}\ed t_0+ \del_{t_k}(\frac{1}{r}\hat{b}^{2n})\ed \ol{t}_0\Big),\quad
 \Big( \del_{\tba_k}\hat{c}^{2n+2}\ed t_0+ \del_{\tba_k}(\frac{1}{r}\hat{b}^{2n})\ed \ol{t}_0\Big),
\]
have no periods on this cycle.
\fi
  
%%%%%%%%%
%%%%%%%%%
%%%%%%%%%
%%%%%%%%%
%%%%%%%%%%%%%%%%%%%%%%%%%%%
\section{Linear finite type surfaces}\label{sec:LFT}
The class of \emph{linear finite type} (ordinary) CMC surfaces are characterized by the property that
a higher-order Jacobi field vanishes, \cite{Pinkall1989},
\be\label{eq:a2N+3}
a^{2N_0+3}=0,\quad N_0\geq 0.
\ee
This implies that, up to scaling by an element in $\C[[\lambda^2]]$,
the formal Killing field $\bY$ factors into a  polynomial Killing field.

In this section, we give a geometric interpretation of this characterization
 in terms of the invariance property of $\bY$  under the higher-order symmetry.

\two
Let $\{ \del_{t_m}\}_{m=0}^{\infty}$ be the frame formally dual to $\{ \ed t_m \}_{m=0}^{\infty}$.
The CMC hierarchy defines a representation of $\{ \del_{t_m}\}_{m=0}^{\infty}$
as a sequence of commuting symmetry vector fields
on $\widehat{\mcm}_{\mcf}$, \S\ref{sec:commusymm}.

For a finite set of constants $\tn{c}_i,   0 \leq i \leq N,$  let
$$\mcv=\sum_{i=0}^N \tn{c}_i\del_{t_i}.$$
The canonical formal Killing field $\bY$  of a    CMC surface 
is stationary with respect to $\mcv$ whenever
\be\label{eq:VY}
\mcv(\bY)=0.
\ee
 
From the initial data \eqref{eq:b2c2} for the coefficients $b^2, c^2$ of $\bY$, 
and the structure equation for $h_2$ in Eq.\eqref{eq:A},
Eq.\eqref{eq:VY}  implies $\mcv(h_2)=0$ and hence
$$\sum_{i=0}^N \tn{c}_i a^{2i+3}=0.$$
It is  known that 
this is equivalent to the linear finite type condition, \cite{Pinkall1989}.

Conversely, since the deformations induced by the CMC hierarchy 
are conformal and preserve Hopf differential,
it is easily checked that,  for a vector field $\mcv$ as above,
$$\mcv (h_2)=0 \;\lra \; \mcv(\bY)=0.$$
It follows that the formal Killing field $\bY$ of a linear finite type CMC surface 
defined by the equation\eqref{eq:a2N+3}
is invariant under the higher-order symmetry  $\mcv=\del_{t_{N_0}}.$

\begin{cor}\label{cor:LFT}
The linear finite type (ordinary) CMC surfaces are characterized by the property
that the canonical formal Killing field $\bY$
is  stationary with respect to a higher-order symmetry.
\end{cor}

This shows that, in a sense, 
the linear finite type CMC surfaces 
generalize such surfaces as Delaunay surfaces and  the twizzlers, 
which are invariant under  a one parameter group of motions 
  of the ambient space form, \cite{Perdomo2012}.

%%%%%%%%%%%%%%%%%%%
%%%%%%%%%%%%%%%%%%%
%%%%%%%%%%%%%%%%%%%
%%%%%%%%%%%%%%%%%%%
%%%%%%%%%%%%%%%%%%%
%%%%%%%%%%%%%%%%%%%

%\np
%-------------------------------------------------------
%\subsection{$t_1$-truncated}\label{sec:t1}

%\bibliographystyle{amsplain}\bibliography{bibliography_today}\end{document}

\begin{thebibliography}{10}

\bibitem{Brendle2013}
S.~Brendle, \emph{Minimal surfaces in {$S^3$}: a survey of recent results},
  Bull. Math. Sciences \textbf{3} (2013), 133--171.

\bibitem{Bryant1995}
Robert~L. Bryant and Phillip~A. Griffiths, \emph{Characteristic cohomology of
  differential systems. {I}. {G}eneral theory}, J. Amer. Math. Soc. \textbf{8}
  (1995), no.~3, 507--596.

\bibitem{DJT2000}
E.~{Date}, M.~{Jimbo}, and T.~{Miwa}, \emph{{Solitons: differential equations,
  symmetries and infinite dimensional algebras. Transl. from the Japanese by
  Miles Reid.}}, Cambridge: Cambridge University Press, 2000.

\bibitem{Dickey2003}
L.~A. Dickey, \emph{Soliton equations and {H}amiltonian systems}, second ed.,
  Advanced Series in Mathematical Physics, vol.~26, World Scientific Publishing
  Co. Inc., River Edge, NJ, 2003.

\bibitem{Falqui1998}
Gregorio {Falqui}, Franco {Magri}, and Marco {Pedroni}, \emph{{Bihamiltonian
  geometry, Darboux coverings, and linearization of the KP hierarchy.}},
  {Commun. Math. Phys.} \textbf{197} (1998), no.~2, 303--324.

\bibitem{Frenkel1995}
Boris {Feigin} and Edward {Frenkel}, \emph{{Kac-Moody groups and integrability
  of soliton equations}}, Invent. Math. \textbf{120} (1995), no.~2, 379--408.

\bibitem{Wang2013}
Daniel Fox and Joe~S. Wang, \emph{Conservation laws for surfaces of constant
  mean curvature in 3-dimensional space forms}, arXiv:1309.6606 (2013).

\bibitem{Frenkel1998}
Edward {Frenkel}, \emph{{Five lectures on soliton equations.}}, {Surveys in
  differential geometry. Vol. IV: Integral systems. Lectures on geometry and
  topology}, Cambridge, MA: International Press, 1998, pp.~131--180.

\bibitem{Gardner1989}
Robert~B. {Gardner}, \emph{{The method of equivalence and its applications.}},
  CBMS-NSF Regional Conference Series in Applied Mathematics, Philadelphia, PA:
  SIAM, 1989.

\bibitem{Gesztesy2000}
Fritz Gesztesy and Helge Holden, \emph{A combined sine-{G}ordon and modified
  {K}orteweg-de {V}ries hierarchy and its algebro-geometric solutions}, AMS/IP
  Stud. Adv. Math. 16. Differential equations and mathematical physics.
  Proceedings of an international conference, Birmingham, AL, USA, March 16-20,
  1999 (2000), 133--173.

\bibitem{Hermann1975}
Robert Hermann, \emph{Sophus {L}ie's 1880 transformation group paper}, Lie
  groups: History, frontiers and applications, vol.~1, Math. Sci. Press, 1975,
  Translated by Michael Ackerman.

\bibitem{Hermann1976}
\bysame, \emph{Sophus {L}ie's 1884 differential invariant paper}, Lie groups:
  History, frontiers and applications, vol.~3, Math. Sci. Press, 1976,
  Translated by Michael Ackerman.

\bibitem{Krasilshchik2004}
P.~Kersten, I.~S. Krasil'shchik, and A.~Verbovetsky, \emph{Nonlocal
  constructions in the geometry of {PDE}}, Proceedings of the fifth
  international conference on symmetry in nonlinear mathematical physics
  \textbf{50} (2004), no.~1, 412--423.

\bibitem{Lie1880}
Sophus Lie, \emph{Theorie der {T}ransformationsgruppen {I}}, Math. Ann.
  \textbf{16} (1880), 441--528.

\bibitem{Lie1884}
\bysame, \emph{Ueber {D}ifferentialinvarianten}, Math. Ann. \textbf{24} (1884),
  537--578.

\bibitem{Mulase1994}
Motohico Mulase, \emph{Algebraic theory of the {KP} equations}, Perspectives in
  mathematical physics, Conf. Proc. Lecture Notes Math. Phys. \textbf{III}
  (1994), 151--217.

\bibitem{Mulase1999}
\bysame, \emph{{Lectures on the asymptotic expansion of a Hermitian
  matrix integral.}}, {Supersymmetry and integrable models. Proceedings of a
  workshop, Chicago, IL, USA, June 12-14, 1997}, Berlin: Springer, 1998,
  pp.~91--134.

\bibitem{Perdomo2012}
Oscar~M. Perdomo, \emph{A dynamical interpretation of the profile curve of
  {CMC} twizzler surfaces}, Pacific Journal of Math. \textbf{258} (2012),
  no.~2, 459--485.

\bibitem{Pinkall1989}
Ulrich Pinkall and Ivan Sterling, \emph{On the classification of constant mean
  curvature tori}, Ann. of Math. (2) \textbf{130} (1989), no.~2, 407--451.

\bibitem{Safronov2013}
Pavel Safronov, \emph{{Virasoro constraints in Drinfeld-Sokolov hierarchies}},
  arXiv:1302.3540 (2013).

\bibitem{Semenov2008}
M.~A. {Semenov-Tian-Shansky}, \emph{{Integrable Systems: the R-matrix
  Approach}}, RIMS-1650 (2008).

\bibitem{Tsujishita1982}
Toru Tsujishita, \emph{On variation bicomplexes associated to differential
  equations}, Osaka J. Math. \textbf{19} (1982), no.~2, 311--363.

\bibitem{Moerbeke1994}
P.~van Moerbeke, \emph{Integrable foundations of string theory}, Lectures on
  integrable systems ({S}ophia-{A}ntipolis, 1991), World Sci. Publ., 1994,
  pp.~163--267.

\bibitem{Vinogradov19841}
A.~M. Vinogradov, \emph{The {C}-spectral sequence, {L}agrangian formalism, and
  conservation laws. {I}. {T}he linear theory}, J. Math. Anal. Appl.
  \textbf{100} (1984), no.~1, 1--40.

\bibitem{Vinogradov19842}
\bysame, \emph{The {C}-spectral sequence, {L}agrangian formalism, and
  conservation laws. {II}. {T}he nonlinear theory}, J. Math. Anal. Appl.
  \textbf{100} (1984), no.~1, 41--129.

\bibitem{Wilson1985}
George {Wilson}, \emph{{Infinite-dimensional Lie groups and algebraic geometry
  in soliton theory.}}, {Philos. Trans. R. Soc. Lond., Ser. A} \textbf{315}
  (1985), 393--404.

\end{thebibliography}

\bibliographystyle{amsplain}
%---------------------------------------------------------------------------------------------------
%---------------------------------------------------------------------------------------------------
\providecommand{\MR}[1]{}
\providecommand{\bysame}{\leavevmode\hbox to3em{\hrulefill}\thinspace}
\providecommand{\MR}{\relax\ifhmode\unskip\space\fi MR }
\providecommand{\MRhref}[2]{%
  \href{http://www.ams.org/mathscinet-getitem?mr=#1}{#2}
}
\providecommand{\href}[2]{#2}

%------------------------------------------------------------------------------------------------
%------------------------------------------------------------------------------------------------
\end{document}